\documentclass[a4,11pt]{article}

\usepackage{amsmath,amssymb,amscd,amsthm}
\usepackage[all]{xy}
\usepackage{enumerate}
\usepackage[dvips]{graphicx}
\usepackage{wrapfig}

\theoremstyle{plain}
\numberwithin{equation}{section}
\newtheorem{thm}{Theorem}[section]
\newtheorem{prop}[thm]{Proposition}
\newtheorem{cor}[thm]{Corollary}
\newtheorem{lem}[thm]{Lemma}

\theoremstyle{definition}
\newtheorem{dfn}[thm]{Definition}
\newtheorem{ex}[thm]{Example}
\newtheorem{rmk}[thm]{Remark}
\renewcommand{\proofname}{\textit{Proof}}

\def\rank{\mathop{\mathrm{rank}}\nolimits}
\def\dim{\mathop{\mathrm{dim}}\nolimits}
\def\Im{\mathop{\mathrm{Im}}\nolimits}
\def\Ker{\mathop{\mathrm{Ker}}\nolimits}

\def\Hom{\mathop{\mathrm{Hom}}\nolimits}
\def\hom{\mathop{\mathrm{hom}}\nolimits}
\def\Ext{\mathop{\mathrm{Ext}}\nolimits}
\def\ext{\mathop{\mathrm{ext}}\nolimits}
\def\<{{\langle}}
\def\>{{\rangle}}

\def\Aut{\mathop{\mathrm{Aut}}\nolimits}
\def\Stab{\mathop{\mathrm{Stab}}\nolimits}
\def\Stabd{\mathop{\mathrm{Stab}^{\dagger}}\nolimits}
\def\+{\mathop{\oplus}\nolimits}

\newcommand{\mf}[1]{{\mathfrak{#1}}}
\newcommand{\mi}[1]{{\mathit{#1}}}
\newcommand{\bb}[1]{{\mathbb{#1}}}
\newcommand{\mca}[1]{{\mathcal{#1}}}
\newcommand{\mr}[1]{{\mathrm{#1}}}

\title{Stability conditions and $\mu$-stable sheaves on K3 surfaces with Picard number one}
\author{Kotaro Kawatani
\footnote{
\textit{2010 Mathematics Subject Classification}: 14F05, 14J28, 18E30.
}
}
\date{}
 
\begin{document}
\maketitle

\begin{abstract}
In this article, we show that some semi-rigid $\mu$-stable sheaves on a projective K3 surface $X$ with Picard number $1$ are stable under Bridgeland's stability condition. 
As a consequence of our work, we show that the special set $U(X) \subset \Stab (X)$ introduced by Bridgeland reconstructs $X$ itself. 
This gives a sharp contrast to the case of an abelian surface. 
\end{abstract}


\section{Introduction and statement of results}

In the paper \cite{Bri}, Bridgeland constructed the theory of stability conditions on triangulated categories $\mca D$. 
Roughly speaking a stability condition $\sigma =(\mca A,Z)$ is a pair consisting of the heart $\mca A$ of a bounded t-structure on $\mca D$ and 
a group homomorphism $Z:K(\mca A) \to \bb C$ 
where $K(\mca A)$ is the Grothendieck group of $\mca A$. 
For $\sigma$, we can define the notion of $\sigma $-stability for objects $E \in \mca D$. 
Very roughly, $E$ is said to be $\sigma$-stable if $\arg Z(A) < \arg Z(E)$ 
for any non-trivial ``subobject'' $A$ of $E$. 
However, there is no notion of subobjects in $\mca D$. 
Thus the heart is necessary for us to define it. 

Let us consider the case $\mca D$ is the bounded derived category $D(X)$ of a projective manifold $X$. 
Namely $D(X)$ is the bounded derived category of $\mathrm{Coh}(X)$, 
where $\mathrm{Coh}(X)$ is the abelian category of coherent sheaves on $X$. 

One of the big problems is 
the non-emptiness of the moduli space $\Stab (\mca D)$ of stability conditions for an arbitrary triangulated category $\mca D$. 
However, when $X$ is a projective K3 surface or an abelian surface, 
Bridgeland found a connected component $\Stabd (X)$ of the space $\Stab(X)$ of stability conditions on $D(X)$. 
$\Stabd (X)$ can be described by using the special locus ``$U(X)$'' given by (See also Sections 2 and 3)  
\begin{multline*}
U(X) := \{ \sigma  \in \Stab (X)|\forall x \in X,\  \mca O_x\mbox{ is $\sigma$-stable with the same phase}\\
\mbox{ and $\sigma$ is good, locally finite and numerical}   \}.
\end{multline*}
Since $U(X)$ is connected by \cite{Bri2}, we can define $\Stabd(X)$ by the connected component which contains $U(X)$. 
We also remark that $U(X)$ is a proper subset of $\Stabd (X)$ if $X$ is a projective K3 surface by \cite{Bri2}. 

Broadly speaking, the topic of our research is an analysis of the relation between 
$U(X)$ and Fourier-Mukai partners of $X$. 
Originally stability conditions are defined on $D(X)$ independently of $X$. 
Let us recall that for some K3 surface $X$, there is another K3 surface $Y$ such that 
$Y$ is not isomorphic to $X$ but $D(Y)$ is equivalent to $D(X)$. 
%
Let $\Phi:D(Y) \to D(X)$ be an equivalence. 
Then $\Phi$ naturally induces an isomorphism $\Phi _* : \Stab(Y) \to \Stab (X)$. 
We shall treat the following problem:
\begin{center}
\textit{Problem.} Suppose that $Y$ is not isomorphic to $X$. 
Then does there exist an equivalence $\Phi :D(Y) \to D(X)$  
so that $\Phi _*(U(Y)) = U(X)$ ? 
\end{center}

We can see that the answer of this problem is negative by the following first main theorem. 

\begin{thm}(Corollary \ref{6.7}) \label{thm1}
Let $X$ and $Y$ be projective K3 surfaces  with Picard number $1$. 
Suppose that $\Phi: D(Y) \to D(X)$ is an equivalence with $\Phi_* (U(Y)) = U(X)$.  
Then $\Phi$ can be written as$:$
\[
\Phi (?) = M \otimes f_* (?) [n], 
\]
where $M$ is a line bundle on $X$, $f$ is an isomorphism $f:Y \to X$ and $n \in \bb Z$. 
\end{thm}

Recall that if $X$ is a projective K3 surface of Picard number $1$ and $Y$ is a projective manifold such that $D(X) \sim D(Y)$ then $Y$ is also a projective K3 surface of Picard number $1$. Suitable reference is, for instance, \cite{B-M} or \cite{Kaw}. 
Furthermore in Corollary \ref{group}, we give the interpretation of Theorem \ref{thm1} from the viewpoint of the autoequivalence group $\Aut (D(X))$ of $D(X)$.  

Theorem \ref{thm1} implies that the special locus $U(X)$ is determined by $X$ although $\Stab(X)$ is defined on the category $D(X)$. 
It is interesting to observe that, 
when $X$ and $Y$ are abelian surfaces, 
$\Phi _* (U(Y)) = U(X)$ for any equivalence $\Phi : D(Y) \to D(X)$ (cf. Remark \ref{6.8}). 
At first, we expected that there exists an equivalence $\Phi : D(Y) \to D(X)$ preserving $U(X)$ although $Y$ is not isomorphic to $X$.

It is well known that any Fourier-Mukai partners of a projective K3 surface $X$ are given by moduli spaces of Gieseker-stable sheaves. 
Hence our first approach was the investigation of $\sigma $-stability of $\mu$-stable (or Gieseker stable) sheaves. 

Before we state the second main theorem Theorem \ref{thm2}, we shall explain two notations which we use in the theorem (the details appear in Section $3$). 
There is a subset $V(X)$ of $U(X)$ which is (roughly) parametrized by $\bb R$-divisors $\beta$ and $\bb R$-ample divisors $\omega$. 
So we write as $\sigma _{(\beta, \omega)} \in V(X)$. 
The set $V(X)$ contains the locus $V(X)_{>2}$ defined by
\[
V(X)_{>2} := \{ \sigma _{(\beta, \omega)} \in V(X) | \omega ^2>2 \}. 
\]

\begin{thm}\label{thm2}
Let $X$ be a projective K3 surface with $\mathrm{NS}(X) = \bb Z \cdot L$. We put $d = L^2 /2$. 
Let $E$ be a torsion free sheaf with $v(E)^2 = 0$ (see section 3.1 for the definition of $v(E)$) and $\rank E \leq \sqrt{d}$, and let $\sigma = (Z,\mca P)$ be in $V(X)_{>2}$.

$(1)$ If $E$ is Gieseker-stable and $E \in \mca P((0,1])$ (see section 2 for the definition of $\mca P((0,1])$), then $E$ is $\sigma $-stable.

$(2)$ If $E$ is $\mu$-stable locally free and $E \in \mca P((-1,0])$ (see section 2 for the definition of $\mca P((-1,0])$), then $E$ is $\sigma $-stable. 

$(3)$ Let $S$ be a spherical sheaf with $\rank S \leq \sqrt{d}$. 
Then $S$ is $\sigma$-stable. 
\end{thm}

The assertions (1) and (2) are proved in Theorem \ref{4.7}, and the assertion (3) is Proposition \ref{5.2}. 
The assumption “$\rank E \leq \sqrt{d}$ is the best possible in some sense (see Example \ref{5.3}), 
and we can not remove the assumption of local-freeness in (2) (see Corollary \ref{5.6}). 
We prove Theorem \ref{thm1} applying Theorem \ref{thm2}.

Finally we explain the contents of this paper.  
Section 2 is a survey of the general theory of stability conditions on triangulated categories. 
In Section 3, we study the case when $\mca D =D(X)$ where $X$ is a projective K3 surface. 
In the last half of Section 3, we shall recall the results on Gieseker stable sheaves and on Fourier-Mukai partners on K3 surfaces with Picard number $1$. 

In Section 4, we shall prove (1) and (2) of Theorem \ref{thm2} (= Theorem \ref{4.7}). 
Hence the main part of this section is the comparison between the $\mu$-stability (or Gieseker-stability) 
and the $\sigma$-stability. 
We remark that 
the $\sigma $-stability of $E \in D(X)$ depends on the argument of the complex number $Z(E)$. 
Hence we need an appropriate description of $Z(E)$ to compare the argument of $Z(E)$ and the slope $\mu_{\omega }(E)$. 
There are two keys for the comparison. 
One is the following expression of the stability function $Z_{(\beta,\omega)}$
(The definition of $Z_{(\beta,\omega)}$ is in Section 3. ) :
\begin{equation*}
Z_{(\beta ,\omega )}(E) = \frac{v(E)^2}{2r_E} + \frac{r_E}{2}
\Bigl( \omega + \sqrt{-1}\bigl( \frac{\Delta _E}{r_E} -\beta    \bigr) \Bigr)^2 . 
\end{equation*}
The other is the assumption that the Picard number of $X$ is one. 
If $X$ satisfies the assumption, the right hand side of the above formula is just complex number. 
Thus we can compare the slope $\mu_{\omega}(E)$ and the argument of $Z(E)$. 

In Section 5, 
we prove Theorem \ref{thm2} (3) (= Proposition \ref{5.2}). 
The strategy of the proof is essentially the same as that of Theorem \ref{4.7}. 
We have two applications of Proposition \ref{5.2}. 
One is to prove that we cannot drop the assumption on rank and  the condition of local-freeness in Theorem \ref{4.7}. 
The other is the determination of Harder-Narashimhan filtrations of some special objects $T_S(\mca O_x)$ (cf. Corollary \ref{5.6} and \ref{5.7}). 
In general, it is very difficult to determine Harder-Narashimhan filtrations. 
So, these examples are valuable. 

In Section 6, we shall treat two applications of Theorem \ref{thm2}. 
The first application is to find some pairs $(E,\sigma)$ such that an object $ E \in D(X)$ is a true complex and 
$E$ is $\sigma $-stable for some $\sigma \in U(X)$. 
The second application is to prove Theorem \ref{thm1}. 
\vspace{11pt}

\noindent
\textbf{Acknowledgement}. 
I am very grateful to the referees 
for their careful reading and  giving me their kind advices and comments. 


\section{Bridgeland's stability condition}
This section is a survey of the general theory of Bridgeland's stability conditions on triangulated categories.
Let $\mca D$ be a $\bb C$ linear triangulated category. The symbol $[1]$ means the shift of $\mca D$ and $[n]$ means 
the $n$-times composition of $[1]$.

\begin{dfn}\label{d1.1}
Let $\sigma = (Z,\mca P)$ be a pair consisting of a group homomorphsim $Z:K(\mca D) \to \bb C$ from the Grothendieck group of $\mca D$ to $\bb C$, and  
a collection $\mca P = \{ \mca P(\phi) \}$ of additive full subcategories $\mca P(\phi)$ of $\mca D$ parametrized by the real numbers $\phi$. 
This pair $\sigma$ is a stability condition on $\mca D$ if it is satisfied the following condition:  \vspace{5pt}\\
(1) If $0 \neq E \in \mca P(\phi ) $, then $Z(E)=m(E) \exp(\sqrt{-1} \pi \phi)$ where $m(E)>0$.  \vspace{5pt}\\
(2) If $\phi > \psi$, then $\Hom_{\mca D}(E,F)=0$ for all $E \in \mca P(\phi)$ and $F \in \mca P(\psi)$. \vspace{5pt}\\
(3) $\mca P(\phi +1) = \mca P(\phi)[1]$. \vspace{5pt}\\
(4) For all $0 \neq E \in \mca D$, there is a sequence of distinguished triangles satisfying the following condition:
\begin{equation}
\resizebox{0.99\hsize}{!}{
\xymatrix{
0\ar[rr]	&	&	E_1\ar[ld]\ar[rr]	& 	& E_2\ar[r]\ar[ld] & \cdots \ar[r] & E_{n-1}\ar[rr]	&	  &E_n=E ,\ar[ld]\\ 
	&A_1\ar@{-->}[ul]^{[1]}& 		&A_2\ar@{-->}[ul]^{[1]}& 	  &		  &			&A_n\ar@{-->}[ul]^{[1]}& 
}
\label{HNF}
}
\end{equation}
where each $A_i$ is in $\mca P(\phi _i)$ ($i=1, \cdots n$) with $\phi _1 > \cdots > \phi _n$.
\end{dfn}

\begin{rmk}\label{r1.2}

(1) Each $\mca P(\phi)$ is an abelian category. 

(2) By definition, for each $0 \neq E \in \mca D$, there is at most one $\phi \in \bb R$ such that $E \in \mca P(\phi)$. 
When $E \in \mca P(\phi)$, 
we define $\arg Z(E) := \phi$ and call $\phi$ the \textit{phase} of $E$. 

(3) $E \in \mca D$ is said to be \textit{$\sigma$-semistable} 
when $E \in \mca P(\phi)$ for some $\phi \in \bb R$.
In particular, if $E$ is minimal in $\mca P(\phi)$ (that is, $E$ has no non-trivial subobjects) 
then $E$ is said to be \textit{$\sigma $-stable}.

(4) The sequence (\ref{HNF}) is unique up to isomorphism. 
We can easily check this by using the property Definition \ref{d1.1} (2).
Hence we define $\phi^+_{\sigma}(E):= \phi _1$, and $\phi ^-_{\sigma }(E) := \phi_n$. 
We call the sequence the \textit{Harder-Narashimhan filtration} (for short HN filtration) of $E$, 
and each $A_i$ a \textit{semistable factor} of $E$.

(5) Let $I \subset \bb R$ be an interval. For $I$, we define 
$\mca P(I)$ as the extension closed additive full subcategory of $\mca D$ generated by $\mca P(\phi)$ ($\phi \in I $). 
If $E \in \mca P(I)$, then $\phi^+(E)$ and $\phi ^-(E) \in I$. 

(6) A stability condition $\sigma $ is said to be \textit{locally finite} if 
for all $\phi \in \bb R$, there is a positive number $\epsilon$ such that the quasi-abelian category 
$\mca P((\phi -\epsilon,\phi +\epsilon))$ is finite length, that is 
both increasing and decreasing sequences of subobjects of $A$ 
will terminate (See also \S 4 of \cite{Bri}). 
The property of local-finiteness guarantees the existence of Jordan-H\"{o}lder filtrations (for short JH filtrations), that is,  
for any $0 \neq A \in \mca P(\phi)$, there exists a sequence of distinguished triangles 
\[
\resizebox{0.99\hsize}{!}{
\xymatrix{
0\ar[rr]	&	&	A_1\ar[ld]\ar[rr]	& 	& A_2\ar[r]\ar[ld] & \cdots \ar[r] & A_{n-1}\ar[rr]	&	  &A_n= A \ar[ld]\\ 
	&S_1\ar@{-->}[ul]^{[1]}& 		&S_2\ar@{-->}[ul]^{[1]}& 	  &		  &			&S_n\ar@{-->}[ul]^{[1]}& 
}
}
\]
such that each $S_i$ is $\sigma $-stable with phase $\phi$.
We call each $S_i $ a \textit{stable factor} of $A$. 
We remark that JH filtrations may not be unique. 
\end{rmk}

In general it is difficult to construct stability conditions on $\mca D$. 
However, by using Proposition \ref{p1.5} (below), we can explicitly construct them in some cases. 
Before we state the proposition, we introduce the notion of a stability condition on abelian categories. 

\begin{dfn}\label{d1.4}
Let $\mca A$ be an abelian category, and 
$Z:K(\mca A) \to \bb C$ a group homomorphism from the Grothendieck group $K(\mca A)$ of $\mca A$ to $\bb C$, 
satisfying 
\[
Z(E) =m_E \exp(\sqrt{-1}\pi \phi_E)\mbox{ for }0\neq E \in \mca A, \mbox{ where }\phi_E \in (0,1]\mbox{ and }m_E >0.
\]
We call $Z$ a \textit{stability function} on $\mca A$. 
An object $E \in \mca A$ is called a \textit{(semi)stable object} for $Z$ when, for any non-trivial subobjects $F$ of $E$, 
the following inequality holds:
\[
\phi _F < \phi _E,  (\phi _F \leq \phi _E). 
\]
If $Z$ has the following property, we call $Z$ a stability function equipped with the \textit{Harder-Narashimhan (for short HN) property}:
\begin{center}
$
0 \neq \forall E \in \mca A, \exists \mbox{ a filtration } 
0 \subset E_1 \subset E_2 \subset \cdots \subset E_{n-1} \subset E_n =E
$ such that \\
$A_i= E_i/E_{i-1}$ is semistable and $\phi _{A_1} > \cdots > \phi _{A_n}$. 
\end{center}
\end{dfn}

\begin{prop}(\cite[Proposition\ 5.3]{Bri}) \label{p1.5}
Let $\mca D$ be a triangulated category. Then the following are equivalent:

(1) To give a stability condition $\sigma =(Z,\mca P)$ on $\mca D$.

(2) To give a pair $(\mca A,Z_{\mca A})$ consisting of  the heart $\mca A$ of a bounded t-structure on $\mca D$ and a stability function $Z_{\mca A}$ on $\mca A$ which has the HN property. 
\end{prop}

For the convenience of readers, we give a sketch of the proof. 
\renewcommand{\proofname}{\textit{From (1) to (2)}}

\begin{proof}
For the pair $\sigma =(Z, \mca P)$, $\mca P((0,1])$ is the heart $\mca A$ of a bounded t-structure on $\mca D$. 
We define a stability function $Z_{\mca A}$ as $Z$. 
Then the pair $(\mca P((0,1]) , Z)$ is what we need. 
\end{proof}
\renewcommand{\proofname}{\textit{From (2) to (1)}} 

\begin{proof}
For a real number $\phi  \in (0,1]$ we define $\mca P(\phi)$ by
\[
\mca P(\phi ):= \{ A \in \mca A | A\mbox{ is semistable for }Z \mbox{ with }\phi _A = \phi \}\cup \{ 0 \} . 
\]
If $\psi \in \bb R \backslash (0,1]$, we define $\mca P(\psi)$ by $\mca P(\psi_0)[k]$ 
where $\psi = \psi_0 +k$ with $\psi _0 \in (0,1]$ and $k \in \bb Z$. 
Since $K(\mca A) = K(\mca D)$, we can define $Z$ by $Z_{\mca A}$. 
Then the pair $(Z,\mca P)$ gives a stability condition on $\mca D$. 
\end{proof}
\renewcommand{\proofname}{\textit{Proof}}

In the following lemma, we introduce two actions of groups on $\Stab (X)$.

\begin{lem}(\cite[Lemma\ 8.2]{Bri})\label{l1.6} 
Let $\Stab (\mca D)$ be the space of stability condition on $\mca D$, 
$\tilde{GL}^+(2,\bb R)$ the universal covering space of ${GL}^+(2,\bb R)$, and 
$\Aut (\mca D)$ the autoequivalence group of $\mca D$. 
$\Stab (\mca D)$ carries a right action of $\tilde {GL}^+ (2,\bb R)$, and a left action of $\Aut (\mca D)$. In addition, these two actions commute. 
\end{lem}

%
%

\begin{rmk}\label{r1.7}
By the definition of the action of $\tilde{GL}^+(2,\bb R)$, we can easily see that 
for any $\sigma \in \Stab(\mca D)$ and any $\tilde g \in \tilde{GL}^+(2,\bb R)$, 
$E \in \mca D$ is $\sigma$-(semi)stable  if and only if 
$E$ is $\sigma \cdot \tilde g$-(semi)stable. 
\end{rmk}

\section{Stability conditions on K3 surfaces}

In this section $X$ is a projective K3 surface over $\bb C$, $\mr{Coh}(X)$ is the abelian category of coherent sheaves on $X$, 
and $D(X)$ is the bounded derived category of $\mr{Coh}(X)$. 
The purpose of this section is to give a description of $\Stab(X)$.

We first introduce some notations. 
Let $A$ and  $B$ be in $D(X)$. 
If the $i$-th cohomology $H^i(A)$ is concentrated only at degree $i=0$, we call $A$ a \textit{sheaf}. 
We put $\Hom ^n_X(A,B) := \Hom_{D(X)}(A,B[n]) $. 
If both $A$ and $B$ are sheaves, then $\Hom_X^n(A,B)$ is just $\Ext^n_{\mca O_X}(A,B)$. 
We also put $\hom^n_{X}(A,B) := \dim _{\bb C}\Hom_X^n(A,B)$ and $\ext_X^n(A,B):= \dim \Ext_{\mca O_X}^n(A,B)$. 
Sometimes we omit $X$ of $\Hom ^n_X (A,B)$ and so on. 
We remark that 
\[
\Hom_X^n(A,B) = \Hom^{2-n}_X(B,A)^*
\]
by the Serre duality. 

We secondly recall the notion of the $\mu$-stability. 
For a torsion free sheaf $F$ and an ample divisor $\omega$, the \textit{slope} $\mu_{\omega}(F)$ is 
defined by $(c_1(F) \cdot \omega)/\rank F$ where $c_1(F)$ is the first Chern class of $F$. 
If the inequality $\mu_{\omega}(A) \leq \mu_{\omega}(F)$ holds for any non-trivial subsheaf $A$ of $F$, then $F$ is said to be \textit{$\mu$-semistable}. 
Moreover if the strict inequality $\mu_{\omega}(A) < \mu_{\omega}(F)$ holds for any non-trivial subsheaf $A$ with $\rank A < \rank F$, then 
$F$ is said to be \textit{$\mu$-stable}. 
The notion of the $\mu$-stability admits the Harder-Narashimhan filtration of $F$ (details in \cite{HL}). 
We define $\mu_{\omega}^+ (F)$ by the maximal slope of semistable factors of $F$, and 
$\mu_{\omega}^- (F)$ by the minimal slope of semistable factors of $F$.

\subsection{On numerical stability conditions on $D(X)$}
  
Let $K(X)$ be the Grothendieck group of $D(X)$. $K(X)$ has the natural $\bb Z$ bilinear form $\chi$:
\[
\chi: K(X) \times K(X) \to \bb Z,\  
\chi( E,F ) := \sum_{i}(-1)^i\hom _X ^i(E,F). 
\]
Let $\mca N(X)$ be the quotient of $K(X)$ by numerical equivalent classes with respect to $\chi$. 
Then $\mca N(X)$ is $H^0(X,\bb Z) \oplus \mr{NS}(X) \oplus H^4(X,\bb Z)$, where 
$\mr{NS}(X)$ is the N\'{e}ron-Severi lattice of $X$. 
A stability condition $\sigma = (Z, \mca P)$ on $D(X)$ is said to be \textit{numerical} if 
$Z$ factors through $\mca N(X)$:
\[
\xymatrix{
K(X) \ar[r]\ar[rd]_{Z} & \mca N(X) \ar[d]^{Z_\mca{N} } \\
            &  \bb C 
}.
\]
Let $\chi _{\mca N}$ be the descent of $\chi$. 
Since $\chi _{\mca N}$ is non-degenerate on $\mca N(X) \otimes _{\bb Z} \bb C$, 
$Z_{\mca N}$ is canonically in $\mca N(X) \otimes \bb C$: 
\[
\Hom _{\bb C}(\mca N(X) \otimes \bb C, \bb C) \to \mca N(X ) \otimes \bb C,\ 
Z_{\mca N} \mapsto Z^{\vee}, 
\]
where $Z(E) = \chi _{\mca N} ( Z^{\vee},E)$. Thus we define $\Stab (X)$ by
\[
\Stab (X) := \{ \sigma \in \Stab (D(X)) | \sigma \mbox{ is locally finite and numerical}  \}.
\]
Then we have the following natural map:
\[
\pi: \Stab (X) \to \mca N(X) \otimes \bb C,\ \pi((Z,\mca P))=Z^{\vee}.
\]
We remark that $\pi$ is a locally homeomorphism (The details are in \cite[Corollary 1.3]{Bri}). 
Hence the map $\pi$ gives a complex structure on $\Stab (X)$. 
In particular $\Stab (X)$ is a complex manifold. 

Let $\langle-,- \rangle$ be the Mukai pairing on $\mca N(X) $: 
\[
\langle r \+ \Delta \+ s, r' \+ \Delta ' \+ s' \rangle = \Delta \Delta ' -r s' -r' s,
\]
where both $r \+ \Delta \+ s$ and $\ r' \+ \Delta ' \+ s'$ are in $H^0(X,\bb Z) \+ \mr{NS}(X) \+ H^4(X,\bb Z)$. 
For an objects $E\in D(X)$, we put $v(E) = ch(E)\sqrt{td_X} \in \mca N(X)$ and call it the \textit{Mukai vector} of $E$. 
Then we have $\chi (E,F) = - \langle v(E),v(F) \rangle$ for $E$ and $F \in D(X)$ by the Riemann-Roch theorem. 
We have the following famous consequence:
\begin{lem}
Let $X$ be a projective K3 surface and $E \in D(X)$. 
Assume that $\hom _X^0(E,E) =1$. Then we have
\[
\< v(E) \>^2 +2 =\hom ^1_X(E,E).
\]
Thus we have $\<  v(E) \> ^2 \geq -2$ and the equality holds if and only if $\hom ^1(E,E)=0$. 
\end{lem}

If, for $E \in D(X)$, $\hom ^1(E,E)=2$, $E$ is said to be \textit{semi-rigid}.  
Assume that $\hom^0(E,E)=1$. 
Then by the above lemma, $\< v(E) \>^2=0$ if and only if $E$ is semi-rigid.

\subsection{Construction of $U(X)$}

Next, following Bridgeland, we define a special subset $U(X)$ of $\Stab (X)$ and give two descriptions of $U(X)$. 
Put 
\[
\mr{NS}(X)_{\bb R} := \mr{NS}(X)\otimes _{\bb Z} \bb R \mbox{ and }
\mr{Amp}(X)_{\bb R} := \{ \omega \in \mr{NS}(X)_{\bb R} | \omega\mbox{ is ample} \}.
\] 
We first define the subset $\mca V(X)$ of $\mr{NS}(X)_{\bb R} \times \mr{Amp}(X)_{\bb R}$ by
\begin{multline*}
\mca V(X):= \{ (\beta,\omega) \in \mr{NS}(X)_{\bb R} \times \mr{Amp}(X)_{\bb R}  | \\ \mbox{ $\forall$ $\delta \in \Delta^+(X)$}, 
\langle \exp ( \beta+ \sqrt{-1}\omega ), \delta  \rangle \not\in \bb R_{\leq 0}\}, 
\end{multline*}
where $\Delta ^+(X)=\{ r\+ \Delta \+ s \in \mca N(X) | \< r\+ \Delta \+ s \>^2 = -2 \mbox{ and }r>0   \}$. 
If $\omega ^2 > 2$ then $(\beta,\omega) \in \mca V(X)$ for all $\beta \in \mr{NS}(X)_{\bb R}$. 
Hence $\mca V(X) \neq \emptyset$. 
Thus we define 
\[
\mca V(X)_{>2} := \{ (\beta,\omega ) \in \mca V(X) | \omega ^2 > 2 \}. 
\]
We can define a torsion pair $(\mca T_{(\beta,\omega)}, \mca F_{(\beta,\omega)})$ (See below) of $\mr{Coh}(X)$ by using a pair $(\beta,\omega) \in \mr{NS}(X)_{\bb R} \times \mr{Amp}(X)_{\bb R}$. 
As a consequence we have a new heart of the bounded t-structure which comes from the torsion pair $(\mca T_{(\beta,\omega)}, \mca F_{(\beta,\omega)})$. 

\begin{lem}\label{2.1}(\cite[Lemma\ 6.1]{Bri2}) 
Let $\beta \in \mr{NS}(X)_{\bb R}$ and $\omega \in \mr{Amp}(X)_{\bb R}$. 
We define respectively $\mca T_{(\beta,\omega )}$, $\mca F_{(\beta,\omega )}$ and 
$\mca A_{(\beta,\omega)}$ by 
\begin{eqnarray*}
\mca T_{(\beta,\omega )} &:=& \{ E \in \mr{Coh}(X) | E\mbox{ is a torsion sheaf or }
\mu _{\omega}^-(E/\mr{torsion} ) > \beta \omega  \},  \\
\mca F_{(\beta,\omega)} &:=& \{ E \in \mr{Coh}(X) | E\mbox{ is torsion free and } 
\mu _{\omega}^+ (E) \leq \beta \omega  \},
\end{eqnarray*}
and
\[
\mca A_{(\beta,\omega )} :=   \{ E^{\bullet} \in D(X)| 
H^i(E^{\bullet})	\begin{cases}
					\in \mca T_{(\beta,\omega)} & (i=0) \\ 
					\in \mca F_{(\beta,\omega)} & (i=-1) \\ 
					= 0 & (i \neq 0,-1) 
					\end{cases}
\}.
\]

$(1)$ The pair $(\mca T_{(\beta,\omega)}, \mca F_{(\beta,\omega)} )$ is a torsion pair of $\mr{Coh}(X)$. 

$(2)$ $\mca A_{(\beta,\omega)}$ is the heart of the bounded t-structure determined by the torsion pair $(\mca T_{(\beta,\omega)}, \mca F_{(\beta,\omega)} )$. 
\end{lem}

The condition that $(\beta,\omega) \in \mca V(X)$ is necessary 
when we construct a stability function $Z_{(\beta,\omega)}$ on $\mca A_{(\beta,\omega)}$. 

\begin{prop}\label{p2.2}(\cite{Bri2}) 
For $(\beta,\omega) \in \mca V(X)$, we define the group homomorphism $Z_{(\beta,\omega)}: K(X) \to \bb C$ by 
\[
Z_{(\beta, \omega )}(E):= \langle \exp (\beta + \sqrt{-1}\omega),v(E)  \rangle .
\]
Then $Z_{(\beta,\omega)}$ is a stability function on $\mca A_{(\beta,\omega)}$ with the HN-property. 
Hence the pair $(\mca A_{(\beta,\omega)} ,Z_{(\beta,\omega )} )$ defines a stability condition $\sigma_{(\beta,\omega)}$ on $D(X)$. 
In particular $\sigma_{(\beta,\omega)}$ is numerical and locally finite. 
\end{prop}


Here we put 
\[
V(X) := \{ \sigma_{(\beta,\omega )} | (\beta,\omega ) \in \mca V(X) \} 
\mbox{ and }
V(X)_{> 2} := \{ \sigma_{(\beta,\omega )} | (\beta,\omega ) \in \mca V(X)_{>2} \}.
\]

The most important property of $\sigma \in V(X)$ is the $\sigma $-stability of the structure sheaves $\mca O_x$ 
of closed points $x$ of $X$.  

\begin{prop}\label{2.3}(\cite[Lemma\ 6.3]{Bri2}) 
Let $x\in X$. 
Then $\mca O_x$ is minimal in $\mca A_{(\beta,\omega)}$ for any $(\beta,\omega ) \in \mca V(X)$. 
Namely $\mca O_x$ does not have non-trivial subobjects in $\mca A_{(\beta, \omega)}$. 
In particular $\mca O_x$ is $\sigma$-stable with phase $1$ for any $\sigma \in V(X)$.  
\end{prop}

\begin{rmk}\label{2.4}
Let $\sigma_{(\beta,\omega)} =(Z,\mca P) \in V(X)$. 

(1) 
By Proposition \ref{2.3} and \cite[Lemma 10.1]{Bri2}, any sheaf $F \in \mr{Coh}(X)$ is in $\mca P((-1,1])$. 
In addition to Proposition \ref{2.3}, if $E \in D(X)$ is $\sigma_{(\beta,\omega)}$-stable with phase $1$ 
then $E$ is $\mca O_x$ for some $x \in X$ or $\mca E[1]$ where $\mca E$ is a locally free sheaf. 
In particular, there is no torsion free $\sigma$-semistable sheaf of phase $1$. 

(2) 
As we stated, $\mr{Coh}(X)$ is a full subcategory of $\mca P((-1,1])$. 
Moreover by Proposition \ref{2.3}, we have 
\begin{equation}
\mca T_{(\beta,\omega)} = \mca P((0,1]) \cap \mr{Coh}(X) \mbox{, and }
\mca F_{(\beta,\omega)} = \mca P((-1,0]) \cap \mr{Coh}(X) \label{torsionpair} . 
\end{equation} 
This fact is proved in Step 2 of the proof of \cite[Proposition 10.3]{Bri2}. 
Now, assume that a torsion free sheaf $E$ is $\mu$-semistable for $\omega$. 
Then by (\ref{torsionpair}):
\[
E \in 	\begin{cases}
					\mca T_{(\beta,\omega)} & (\mbox{if }\mu_{\omega}(E) > \beta \omega) \\
					\mca F_{(\beta,\omega)} & (\mbox{if }\mu_{\omega}(E) \leq \beta \omega ) . 
		\end{cases}
\]
\end{rmk}

We define
\[
U(X) := V(X) \cdot \tilde{GL}^+(2,\bb R)
\mbox{ and }
U(X)_{>2} := V(X)_{>2} \cdot \tilde{GL}^+(2,\bb R).
\]
We remark that the action of $\tilde {GL}^+(2,\bb R)$ on $U(X)$ is transitive. 
Since $V(X)$ is connected, $U(X)$ is also connected. 
This is the concrete definition of $U(X)$. 
Conversely we shall give an abstract definition of $U(X)$. 
To do this, we define the notion of good stability conditions. 

For $\mho \in \mca N(X)\otimes \bb C$, we have $\mho = \mho _R + \sqrt{-1} \mho _{I}$ where $\mho _{R}$ and $\mho _I $ are in $\mca N(X)\otimes \bb R$. 
Let $P(X)$ be the set of vectors $\mho \in \mca N(X) \otimes \bb C$
such that Mukai pairing is positive definite on the real $2$-plane spanned by $\mho_{R}$ and $\mho _{I}$. 
Let $\Delta(X)$ be the subset of $\mca N (X)$ defined by
\[
\Delta (X) := \{ \delta \in \mca N(X) | \< \delta \> ^2 =-2  \}. 
\]
We define $P_0(X)$ by 
\[
P_0(X) := P(X)-\bigcup_{\delta \in \Delta(X) }\delta ^{\perp},
\]
 where $\delta^{\perp}= \{ \mho \in \mca N(X) \otimes \bb C | \< \mho, \delta \> =0  \}$. 

\begin{dfn}\label{2.5}
A stability condition $\sigma \in \Stab (X)$ is said to be \textit{good}, if 
$\pi (\sigma) \in P_0(X)$. 
\end{dfn}

\begin{prop}\label{2.6}(\cite[Proposition\ 10.3]{Bri2}) 
We have 
\[
U(X) =\{ \sigma \in \Stab (X) | \sigma \mbox{ is good and }\forall \mca O_x\mbox{ is $\sigma$-stable in a common phase}.  \}.
\]
\end{prop}

In \cite{Bri2}, $U(X)$ is defined by the right hand side of Proposition \ref{2.6}. 
Define $\Stabd (X)$ by the unique connected component containing $U(X)$.

\subsection{Gieseker stability and Fourier-Mukai partners}

The last topic of Section 3 is a review of Gieseker stability. 
The details are in \cite{HL}. 
Let $E$ be a torsion free sheaf on a K3 surface $X$ 
and $p(E)$ the reduced Hilbert polynomial for an ample divisor $L$:
\[
p(E) = \frac{\chi (\mca O_X, E \otimes n L)}{\rank E} = \frac{\chi (-n L,E)}{\rank E} \in \bb Q[n]  .
\]
Using the Mukai vector $v(E)= r_E \+ \Delta _E \+ s_E$ of $E$, we write down $p(E)$:
\begin{eqnarray}
p(E)	&=&	-\frac{ \< v(-n L) ,v(E) \>  }{r_E} \notag \\
		&=&	\frac{L^2}{2}n^2 + \frac{\Delta. L}{r_E}n + \frac{s_E}{r_E} +1\label{Hilbert}.
\end{eqnarray}

A torsion free sheaf $E$ is called a \textit{Gieseker semistable} sheaf 
if, for any non-trivial subsheaf $A$, $p(A) \leq p(E)$ as polynomial. 
In particular, $E$ is called a \textit{Gieseker stable} sheaf when the strict inequality $p(A) < p (E)$ holds. 
For a torsion free sheaf $E$, we can easily check the following well known fact by the formula (\ref{Hilbert}):
\begin{center}
$\mu $-stable $\Rightarrow $ Gieseker stable $\Rightarrow $ Gieseker semistable $\Rightarrow $ $\mu$-semistable.
\end{center}

Let $\mca M_{L}(v)$ be the moduli space of Gieseker stable torsion free sheaves 
with Mukai vector $v = r \+ \Delta \+ s$. 
If $v$ is primitive in $\mca N(X)$, then $\mca M_{L}(v)$ is projective. 

By the result of \cite{HLOY} or \cite{Ste}, 
we have a beautiful description of Fourier-Mukai partners of $X$ 
when the Picard number of $X$ is $1$. 
Let us recall it.

\begin{thm}(\cite[Theorem\ 2.1]{HLOY}, \cite{Ste})\label{3.3}
Let $X$ be a projective K3 surface with $\mr{NS}(X) = \bb Z \cdot L$ where $L$ is an ample line bundle on $X$, and 
let $\mr{FM}(X)$ be the set of isomorphic classes of Fourier-Mukai partners of $X:$
\[
\mr{FM}(X) = \{ Y | Y\mbox{ is a projective K3 surface and }D(Y) \sim D(X) \} /\sim _{\mr{isom}}.
\]
Then $\mr{FM}(X)$ is given by 
\[
\mr{FM}(X) =\{ \mca M _L (r \+ L \+ s) | 2rs = L^2,\ \gcd (r,s)=1,\ r\leq s  \}. 
\]
\end{thm}

We remark that $\mca M _L (r \+ L \+ s)$ is the fine moduli space of $\mu$-stable sheaves, since 
$\mr{NS}(X) = \bb Z \cdot L$.

\section{$\sigma$-stability of $\mu$-stable semi-rigid sheaves}

From this section we mainly consider projective K3 surfaces with Picard number $1$. 
In this article, a pair $(X,L)$ is said to be a \textit{generic K3}, 
if $X$ is a projective K3 surface and $L$ is an ample line bundle which generates $\mr{NS}(X)$. 
We define $\deg X$ by $L^2$ and call it \textit{degree of $X$}. 
We also write the Mukai vector $v(E)$ of $E \in D(X)$ by $r_E \+ \Delta _E \+ s_E $. 
Then we have $r_E = \rank E$, $\Delta _E = c_1(E)$ and $s_E = \chi (\mca O_X,E)-\rank E$. 
Since $\mr{NS}(X) = \bb Z \cdot L$, we can write $\Delta _E  = n_E L$ for some integer $n_E \in \bb Z$. 
So we also write $v(E) = r_E \+ n_E L \+ s_E$. 
 
Our research and results are based on another expression of the function $Z_{(\beta ,\omega)}$, 
where $\sigma _{(\beta ,\omega )} = (Z_{(\beta, \omega )},\mca P_{(\beta,\omega)}) \in V(X)$. 
For $E \in D(X)$, assume that $r_E \neq 0$. 
Then we can rewrite the stability function $Z_{(\beta,\omega)}$ in the following way
\footnote{We wrote the symbols $\<, \>$ till last section. 
From here we will omit them. }:
\begin{equation}
Z_{(\beta ,\omega )}(E) = \frac{v(E)^2}{2r_E} + \frac{r_E}{2}
\Bigl( \omega + \sqrt{-1}\bigl( \frac{\Delta _E}{r_E} -\beta    \bigr) \Bigr)^2 
\label{4.0}.
\end{equation}

We introduce a function which will appear in the proofs of Lemmas \ref{4.6} and \ref{5.1}, and in Example \ref{5.3}. 
For a generic K3 $(X,L)$ with degree $2d$, assume that $\sigma_{(\beta, \omega)}= (Z_{(\beta, \omega)}, \mca P_{(\beta, \omega)}) \in V(X)$. We put $(\beta, \omega) =(xL, yL)$. 
Then, for $E \in D(X)$, the imaginary part of $Z_{(\beta,\omega)}(E)$ is 
$2\sqrt{-1} yd \lambda _E $ where $\lambda _E= n_E -r_E x $. 
For $E, A \in D(X)$, we define $N_{A,E}(x,y)$ by
\begin{equation}
N_{A,E}(x,y):=\lambda _E \cdot \mf{Re}Z_{(\beta,\omega)}(A)  - \lambda _A \cdot \mf{Re}Z_{(\beta,\omega)}(E), \label{NAE}
\end{equation}
where $\mf{Re}$ means taking the real part. 

Recall the notion $\arg Z(A) $ for a $\sigma $-semistable object $A$ and $\sigma \in \Stab (X)$ (cf. Remark \ref{r1.2} (2)). 
In general, we can not determine the argument of the complex number $Z(E)$ for an object $E \in D(X)$. 
However if $E \in \mca P((a,a+1 ])$ (for some $a \in \bb R$) then we can determine the argument of $Z(E)$. 
So we denote also it by $\arg Z(E)$, that is, 
$ \phi = \arg Z(E) \stackrel{\mr{def}}{\iff} Z(E)= m \exp (\sqrt{-1}\pi \phi)$ 
for some $m \in \bb R_{>0}$ .

We shall use Lemma \ref{4.1} and Proposition \ref{4.2} to analyze of the maximal (semi)stable factor of Gieseker stable sheaves $E$ when $E \in \mca P((0,1])$ for $\sigma =(Z,\mca P) \in V(X)$. 

\begin{lem}\label{4.1}
Let $(X,L)$ be a generic K3 and $\sigma_{(\beta,\omega)} = (Z, \mca P) \in V(X)$. 
Assume that $A \to E \to F \to A[1]$ is a non-trivial distinguished triangle in $\mca P((0,1])$, 
that is, $A$, $E$ and $F $ are in $\mca P((0,1])$. 

$(1)$ If $E$ is a torsion free sheaf then $A$ is also a torsion free sheaf. 

$(2)$ In addition to $(1)$, assume that $E$ is a Gieseker stable sheaf. 
If $\arg Z(E) \leq \arg Z(A) < 1 $, then $\mu_{\omega }(A) < \mu_{\omega}(E)$. 
\end{lem}

\begin{proof}
We first prove the assertion $(1)$. 
If $G \in \mca P((0,1]) = \mca A_{(\beta,\omega)}$, then the $i$-th cohomology $H^i(G)$ is concentrated at $i=0$ and $-1$. 
Then we see that $A$ is a sheaf by the exact sequence
\[
\begin{CD}
0= H^{-2}(F) @>>> H^{-1}(A) @>>> H^{-1}(E) =0
\end{CD}
\]
where we use the fact that $E$ is a sheaf for the last equality. 
Since $E$ and $A$ are sheaves, we have the following exact sequence of sheaves:
\[
\begin{CD}
0 @>>> H^{-1}(F) @>>> A @>f>> E @>>> H^0(F) @>>> 0 .
\end{CD}
\]
The sheaf $H^{-1}(F)$ is torsion free since it is in $\mca F_{(\beta,\omega)}$. 
Thus $A$ is an extension of torsion free sheaves. 
Hence $A$ is torsion free. 
\vspace{5pt} 

Let us prove the assertion $(2)$. \\
\textit{Case I. When $H^{-1}(F) = 0$. }

Then $A$ is a subsheaf of $E$. 
So we have 
\begin{equation}
p(A) < p(E). \label{Hilb}
\end{equation}
Thus $\mu_{\omega}(A) \leq \mu _{\omega}(E)$. Assume that $\mu_{\omega}(A) = \mu_{\omega}(E)$. 
By the formula (\ref{Hilbert}) and the inequality (\ref{Hilb}) we have 
\[
\frac{s_A}{r_A} < \frac{s_E}{r_E}, 
\]
where $v(A) = r_A \+ \Delta _A \+ s_A$ and $v(E) = r_E \+ \Delta _E \+ s_E$. 
Hence we have $v(A)^2/r_A^2 > v(E)^2/r_E^2$. 
Here we also used the fact that the Picard number is $1$. 
Combining this with $\mu_{\omega}(A) = \mu_{\omega }(E)$, 
we have $\arg Z(A)/r_A < \arg Z(E)/r_E$ by the formula (\ref{4.0}). 
This contradicts the fact that $\arg Z(E) \leq  \arg Z(A)$. 
\vspace{5pt}

\noindent
\textit{Case II. When $H^{-1}(F) \neq 0$. }

Recall that $H^{-1}(F)$ is torsion free. 
We have the following inequalities:
\[
\mu_{\omega}(H^{-1}(F)) \leq \mu_{\omega}^+(H^{-1}(F)) \leq \beta \omega < 
\mu_{\omega}^-(A) \leq \mu_{\omega}(A). 
\]
Hence we have $\mu_{\omega }(H^{-1}(F)) < \mu_{\omega}(A) < \mu_{\omega}(\Im (f))$, 
where $\Im (f)$ is the image of $f:A \to E$. 
Since $\Im (f)$ is a subsheaf of $E$, 
$\mu_{\omega }(\Im (f)) \leq \mu_{\omega }(E)$. 
Hence we have $\mu_{\omega }(A) < \mu_{\omega}(E)$. 
\end{proof}

As a consequence of Lemma \ref{4.1}, we prove the following proposition. 

\begin{prop}\label{4.2}
Let $(X,L)$ be a generic K3, let $\sigma = \sigma_{(\beta,\omega)} =(Z,\mca P)$ be in $V(X)$, 
and let $E$ be a Gieseker stable torsion free sheaf with $v(E)^2 \leq 0$ and $E \in \mca P((0,1])$.

$(1)$ Assume that $E$ is not $\sigma $-semistable. Then there is a torsion free $\sigma$-stable sheaf $S$ such that 
$\beta \omega < \mu_{\omega}(S) < \mu_{\omega}(E)$, $v(S)^2=-2$ and $\arg Z(S) =\phi_{\sigma}^+(E)$. 
In particular $\arg Z(E) < \arg Z(S)$. 

$(2)$ Assume that $E$ is not $\sigma $-stable but $\sigma$-semistable. 
Then there is a torsion free $\sigma$-stable sheaf $S$ 
such that $\beta \omega < \mu_{\omega}(S) < \mu_{\omega}(E)$, $v(S)^2=-2$ and $\arg Z(S)= \arg Z(E)$. 
\end{prop}

\begin{proof}
We prove $(1)$. 
Since $E$ is not $\sigma $-semistable, there is the non-trivial HN-filtration of $E$:
\[
\resizebox{0.99\hsize}{!}{
\xymatrix{
0\ar[rr]	&	&	E_1\ar[ld]\ar[rr]	& 	& E_2\ar[r]\ar[ld] & \cdots \ar[r] & E_{n-1}\ar[rr]	&	  &E_n=E \ar[ld]\\ 
	&A_1\ar@{-->}[ul]^{[1]}& 		&A_2\ar@{-->}[ul]^{[1]}& 	  &		  &			&A_n\ar@{-->}[ul]^{[1]}& 
}}.
\]
Let $S$ be a stable subobject of $A_1$. 
We show that $S$ satisfies our requirement. 
By the composition of natural two morphisms, 
we have the following distinguished triangle in $\mca P((0,1])$:
\begin{equation}
\begin{CD}
S @>>> E @>>> F @>>> S[1] \label{sankaku}
\end{CD}.
\end{equation}
Then $S$ is a torsion free sheaf by Lemma \ref{4.1} (1). 
By Remark \ref{2.4}, we have $\arg Z(S)= \arg Z(A_1) < 1$. 
Thus $\beta \omega < \mu_{\omega}(S)$. 
By Lemma \ref{4.1}, $\mu _{\omega}(S) < \mu _{\omega }(E)$. 
Hence $v(S)^2$ should be negative by 
the assumption $v(E)^2 \leq 0$ and the formula (\ref{4.0}). 
Since $S$ is stable, we have $v(S)^2=-2$. 
\vspace{5pt}

Next we prove $(2)$. 
If $E$ satisfies the assumption, $E$ has a $\sigma $-stable subobject $S$ with $\arg Z(S)= \arg Z(E)$. 
Thus we have the same triangle as (\ref{sankaku}). Hence we have proved the assertion. 
\end{proof}

%

Next we prepare, in some sense, dual assertions of Lemma \ref{4.1} and Proposition \ref{4.2} for the case $E \in \mca P((-1,0])$. 

\begin{lem}\label{4.4}
Let $(X,L)$ be a generic K3 and $\sigma _{(\beta ,\omega)} = (Z,\mca P) \in V(X)$. 
Assume that $F \to E \to A \to F[1]$ is a non-trivial distinguished triangle in $\mca P((-1,0])$. 

$(1)$ If $E$ is a torsion free sheaf then $A$ is also a torsion free sheaf.

$(2)$ If $E$ is a $\mu$-stable locally free sheaf, then $A$ is a torsion free sheaf and 
the strict inequality $\mu_{\omega}(E) < \mu_{\omega}(A)$ holds. 
\end{lem}

\begin{proof}We first prove $(1)$. 
Since $\mca P((-1,0]) = \mca P((0,1])[-1] = \mca A_{(\beta ,\omega)}[-1]$, 
the $i$-th cohomology $H^i(G)$ of $G \in \mca P((-1,0])$ is concentrated at $i=0$ and $1$. 
Note that $H^1(A)=$ is $0$ by the fact $H^2(F) = H^1(E)=0$. 
Since $E$ and $A$ are sheaves, we have the following exact sequence of sheaves:
\[
\begin{CD}
0 @>>> H^0(F) @>>> E @>f>> A @>>>H^1(F) @>>> 0 .
\end{CD}
\]
Since $A \in \mca F_{(\beta,\omega)}$, $A$ is torsion free. 
We remark that $H^0(F)$ is also torsion free. 
\vspace{5pt}

Next we prove the inequality in $(2)$. \\
\textit{Case I. When $H^0(F) \neq 0$. }

Then $\rank (\Im (f)) < \rank E$ where $\Im (f)$ is the image of $f$. 
Since $E$ is $\mu$-stable, 
we have $\mu_{\omega}(E) < \mu_{\omega}(\Im (f))$. \\
(I-i) Assume that $H^1(F)=0$. Then $\Im(f)=A$. So we have $\mu_{\omega}(E) < \mu_{\omega }(A)$. \\
(I-ii) Assume that $H^1(F)$ is torsion. Then $\omega \Delta_{H^1(F)} \geq 0$. 
Since $\rank \Im(f) = \rank A$ and $\Delta _A = \Delta _{\Im (f)} + \Delta _{H^1(F)}$, we have 
$\mu_{\omega }(\Im (f)) \leq \mu_{\omega}(A)$. Hence we get the inequality. \\
(I-iii) Assume that $ H^1(F) \supsetneqq T$, where $T$ is the maximal torsion subsheaf of $H^1(F)$. 
Then we have the following diagram of exact sequences:
\[
\begin{CD}
@.		@.		A	\\
@.		@.		@VVV\\
0 @>>> T @>>> H^1(F) @>>> H^1(F)/T @>>>0 \\
@.		@.		@VVV\\
@.		@.		0	
\end{CD}.
\]
Recall the following inequalities:
\[
\mu_{\omega}(A) \leq \mu_{\omega}^+ (A) \leq \beta  \omega < \mu_{\omega}^-(H^1(F)/T) \leq \mu_{\omega}(H^1(F)/T) .
\]
By the argument of (I-ii), we have $\mu_{\omega}(H^1(F)/T) \leq \mu_{\omega}(H^1(F))$. 
So $\mu_{\omega}(A) < \mu_{\omega} (H^1(F))$.  
Since the following sequence is exact, we have $\mu_{\omega}(\Im (f)) < \mu_{\omega}(A)$:
\[
\begin{CD}
0@>>> \Im (f) @>>> A @>>> H^1(F) @>>> 0.
\end{CD}
\]
Thus we have proved the inequality $\mu_{\omega}(E) < \mu_{\omega}(A)$. \vspace{5pt}\\
\textit{Case II. When $H^0(F)=0$. }

The sequence 
\begin{equation}
\begin{CD}
0 @>>> E @>>> A @>>> H^1(F) @>>> 0
\end{CD} \label{exact}
\end{equation}
is an exact sequences of sheaves. Hence we use $F$ instead of $H^1(F)$. 
Notice that both $A$ and $E$ are in $\mca F_{(\beta,\omega)}$ and that $F$ is in $ \mca T_{(\beta,\omega)}$. \\
(II-i) Assume that $F \supsetneqq \mi{tor} $ where $\mi{tor}$ is the maximal torsion subsheaf of $F$.  
By the argument of (1-iii), we have the inequality. \\
(II-ii) Assume that $F$ is torsion with $\dim \mr{Supp}( F) =1$. 
Then $\rank A = \rank E$ and $\Delta _F \omega >0$. So we have the inequality. \\
(II-iii) Assume that $F$ is torsion with $\dim \mr{Supp}(F) =0$. 
Let $x$ be a closed point  in $\mr{Supp}(F)$. 
By (\ref{exact}), 
we have the exact sequence of $\bb C$ vector spaces:
\begin{center}
\resizebox{0.99\hsize}{!}{
$
\begin{CD}
\Ext ^1_{\mca O_X}(E,\mca O_x) @>>> \Ext_{\mca O_X}^2(F, \mca O_x) @>>> \Ext_{\mca O_X}^2(A, \mca O_x) @>>> \Ext_{\mca O_X}^2(E,\mca O_x)
\end{CD}
$.
}
\end{center}
\vspace{5pt}
Since $E$ is locally free and $\dim X=2$, $\Ext ^1_{\mca O_X}(E,\mca O_x)=\Ext ^2_{\mca O_X}(E,\mca O_x)=0$. 
By the Serre duality we have 
\[
\Ext ^2_{\mca O_X}(F,\mca O_x) = \Hom _X^0(\mca O_x, F)^* 
\mbox{ and }
\Ext ^2_{\mca O_X}(A,\mca O_x) = \Hom _X^0 (\mca O_x, A)^*.
\]
Since $x \in \mr{Supp}(F)$, $\Hom_{X}^0(\mca O_x,F) \neq 0$. 
So $\Hom_{X}^0(\mca O_x,A)$ also is not $0$. This contradicts the torsion-freeness of $A$. 
Thus we complete the proof. 
\end{proof}

\begin{prop}\label{4.5}
Let $(X,L)$ be a generic K3, let $\sigma=(Z, \mca P)$ be in $V(X)$, 
and let $E$ be a $\mu$-stable locally free sheaf with $v(E)^2 \leq 0$ and $E \in \mca P((-1,0])$. 

$(1)$ Assume that $E$ is not $\sigma$-semistable. 
Then there is a $\sigma$-stable torsion free sheaf $S$ such that 
$\mu_{\omega}(E) < \mu_{\omega}(S)$, $v(S)^2= -2$ and $\arg Z(S) = \phi_{\sigma}^{-}(E)$. 
In particular $\arg Z(S) < \arg Z(E)$ and $\mu_{\omega}(S) <\beta \omega$. 

$(2)$ Assume that $E$ is not $\sigma$-stable but $\sigma $-semistable. 
Then there is a $\sigma$-stable torsion free sheaf 
$S$ such that $\mu_{\omega}(E) < \mu_{\omega}(S)$, $v(S)^2=-2$ and $\arg Z(E) = \arg Z(S)$. 
Moreover we have $\mu_{\omega }(S) < \beta \omega$. 
\end{prop}

\begin{proof}
Let us prove $(1)$. 
Since $E$ is not $\sigma$-semistable, $E$ has the HN-filtration:
\[
\resizebox{0.99\hsize}{!}{
\xymatrix{
0\ar[rr]	&	&	E_1\ar[ld]\ar[rr]	& 	& E_2\ar[r]\ar[ld] & \cdots \ar[r] & E_{n-1}\ar[rr]	&	  &E_n=E \ar[ld]\\ 
	&A_1\ar@{-->}[ul]^{[1]}& 		&A_2\ar@{-->}[ul]^{[1]}& 	  &		  &			&A_n\ar@{-->}[ul]^{[1]}& 
}
}.
\]
Let $S$ be a stable quotient of $A_n$ in $\mca P((-1,0])$. 
Then we show that $S$ is what we need. 
By the composition of natural morphisms, we have the following distinguished triangle in $\mca P((-1,0])$:
\begin{equation}
\begin{CD}
F @>>> E @>>> S @>>> F[1] \label{sankaku2}
\end{CD}.
\end{equation}

By Lemma \ref{4.4}, $S$ is a torsion free sheaf and we have $\mu_{\omega}(E) < \mu_{\omega }(S)$. 
Since $v(E)^2 \leq 0$, $v(S)^2$ should be negative. 
Since $S$ is $\sigma$-stable, we have $v(S)^2=-2$. 
Finally we prove the inequality $\mu_{\omega}(S) < \beta \omega$. 
Since $S \in \mca P((-1,0])$ we have $\mu_{\omega}(S)\leq \mu_{\omega}(S)^+ \leq \beta \omega$. 
So, If the equality $\mu_{\omega}(S) =\beta \omega$ holds then we have $\arg Z(S)=0$. 
This contradicts the fact that $\arg Z(S) < \arg Z(E) \leq 0$. 
\vspace{7pt} 

(2) By the assumption, $E$ has a stable quotient $E \to S$. Then we have the same triangle as (\ref{sankaku2}). 
Similarly to $(1)$ we see that $S$ is a $\sigma$-stable torsion free sheaf with $v(S)^2=-2$ and $\mu_{\omega}(E) < \mu_{\omega}(S)$. 
Finally we consider the inequality $\mu_{\omega}(S) < \beta \omega$. 
Similarly to $(1)$,  we have $\mu_{\omega}(S) \leq \beta \omega $. 
If $\mu_{\omega}(S) =\beta \omega$ then $\arg Z(S)=0$. 
On the other hand, we have $\mu_{\omega}(E) < \mu_{\omega}(S) = \beta \omega$. Thus $\arg Z(E)$ should be negative. 
This contradicts the fact that $\arg Z(E) =\arg Z(S)$. 
Thus we have got the assertion. 
\end{proof}

The following lemma is very important 
since it implies the non-existence of $\sigma$-stable factors in the proof of Theorem \ref{4.7}. 

\begin{lem}\label{4.6}
Let $(X,L)$ be a generic K3 with $\deg X = 2d$. 
Assume that $E$ is a sheaf with $0< \rank E \leq \sqrt{d}$ and $v(E)^2=0$, 
and $A$ is a sheaf with $v( A) ^{2}=-2$. 
For $\sigma _{(\beta,\omega)} =(Z,\mca P) \in V(X)_{>2}$, the following holds. 

$(1)$ If $\beta \omega < \mu _{\omega}( A) < \mu_{\omega}(E)$, then $0< \arg Z(A) < \arg Z(E) < 1$. 

$(2)$ If $\mu_{\omega }(E) < \mu_{\omega}(A) < \beta \omega$, then $-1 < \arg Z(E) < \arg Z(A) < 0$. 

\end{lem}

\begin{proof}
Since $\mr{NS}(X) = \bb Z\cdot L$, we put 
\begin{center}
$\beta =x L$, $\omega =y L$, 
$v(E)= r_E \+ n_E L \+ s_E$ and $v(A ) = r_A \+ n_A L \+ s_A$. 
\end{center}
Since $v(A)^2 =-2$, $r_A$ is positive. 
By the formula (\ref{4.0}) and by the fact $v(E)^2=0$, we have 
\begin{eqnarray}
Z(E)	&=&	\frac{r_E}{2}\Bigl( \omega + \sqrt{-1}\bigl( \frac{n_E L}{r_E} -\beta \bigr) \Bigr)^2	\notag \\
		&=& dr_E \Bigl( y^2 - \frac{\lambda _E ^2}{r_E^2} \Bigl) + 2\sqrt{-1}dy\lambda _E  \notag ,
\end{eqnarray}
where $\lambda _E=n_E - r_E x$, and 
\begin{eqnarray}		
Z(A)	&=& \frac{v(A)^2}{2 r_A}+ \frac{r_A}{2}\Bigl( \omega + \sqrt{-1}\bigl( \frac{n_A L}{r_A} -\beta \bigr) \Bigr)^2 \notag \\
		&=& -\frac{1}{r_A}+ dr_A \Bigl( y^2 - \frac{\lambda _A ^2}{r_A^2} \Bigl) + 2\sqrt{-1}dy\lambda _A \notag, 
\end{eqnarray}
where $\lambda _A= n_A -r_A x$. \vspace{5pt} \\
\textit{The proof of (1).}

By the assumption, we have $x < \frac{n_A }{r_A} < \frac{n_E}{ r_E}$. 
So both $\lambda _A$ and $\lambda _E$ are positive, and the strict inequality $r_A n_E - r_E n_A > 0$ holds. 
Hence 
\begin{eqnarray*}
\arg Z(A) < \arg Z(E)	& \iff & \frac{\mf{Re} Z(E) }{\lambda _E}  < \frac{\mf{Re}Z(A)}{\lambda _A}  \\
						& \iff & 0< N_{A,E}(x,y) .
\end{eqnarray*}
Then 
\begin{eqnarray}
N_{A,E}(x,y)	&=& \lambda _E \Bigl( -\frac{1}{r_A}+ d r_A y^2 - \frac{d \lambda _A^2}{r_A}  \Bigr) 
					-\lambda _A \Bigl (  d r_E y^2 - \frac{d \lambda _E^2}{r_E}  \Bigr) \notag \\
				&=& d y^2 (r_A \lambda _E -r _E \lambda _A) + d \lambda _A \lambda _E \Bigl( \frac{\lambda _E}{r_E} - \frac{\lambda _A }{r_A} \Bigr)
					-\frac{\lambda _E }{r_A} \notag \\
				&=& d y^2 (r_A n_E - r_E n_A) + d(n_A - r_A x)(n_E - r_E x)
					\Bigl( \frac{n _E}{r_E} - \frac{n _A }{r_A} \Bigr) \notag \\
				& &	-\frac{n_E-r_E x}{r_A} \notag \\
				&=& d(r_A n_E - r_E n_A)y^2 + d(r_A n_E - r_E n_A)(x-\mf{a}) ^2 \notag \\ 
				& &- d(r_A n_E - r_E n_A) \mf{a}^2+ d\frac{n_A n_E}{r_A r_E}(r_A n_E - r_E n_A)-\frac{n_E}{r_A}, \label{N} 
\end{eqnarray}
where 
\[
\mf{a}:=\frac{1}{2}\Bigl( \frac{n_A}{r_A}+ \frac{n_E}{r_E}-
		\frac{r_E}{d r_A (r_A n_E - r_E n_A) }\Bigr). 
\]

We shall prove $N_{A,E}(x,y) > N_{A,E}(\frac{n_A}{r_A}, \frac{1}{\sqrt d})$ (notice that $y^2 =\frac{1}{d} \iff \omega ^2 =2$) for any $(\beta, \omega)$ satisfying the assumption. 
We first prove $\frac{n_A}{r_A} \leq \mf{a}$. 
In fact, 
\begin{eqnarray}
\frac{n_A}{r_A} \leq \mf{a}	&\iff & \frac{n_A}{r_A}- \frac{n_E}{r_E} \leq \frac{r_E}{d r_A (r_E n_A -r_A n_E)} \notag \\
							&\iff & \frac{r_E n_A -r_A n_E}{r_E} \leq \frac{r_E}{d(r_E n_A -r_A n_E)} \label{last} 		
\end{eqnarray}
Since the integer $r_E n_A - r_A n_E$ is smaller than $0$, the inequality (\ref{last}) is equivalent to the following:
\begin{equation}
\frac{(r_E n_A - r_A n_E )^2}{r_E^2} \geq \frac{1}{d}. 
\label{iikae}
\end{equation}
Since $(r_E n_A- r_A n_E)^2 > 0$ and $\sqrt d \geq r_E$, the inequality (\ref{iikae}) holds. 
Hence we have $\frac{n_A}{r_A} \leq  \mf {a}$. 

Since $(r_A n_E- r_E n_A)>0$, $N_{A,E}(x,y)$ is strict increasing with respect to $y> 1/\sqrt{d}$. 
Since $(r_A n_E- r_E n_A)>0$ and $x < \frac{n_A}{r_A} \leq \mf{a}$, 
$N_{A,E}(x,y)$ is strict decreasing with respect to $x< \frac{n_A}{r_A}$. 
Hence we have $N_{A,E}(x,y) > N_{A,E}(\frac{n_A}{r_A}, \frac{1}{\sqrt d})$. 

If we prove $N_{A,E}(\frac{n_A}{r_A},y) > 0$, the proof will be complete. 
If $x= \frac{n_A}{r_A}$, we have $N_{A,E}(x,y)= \lambda _E \cdot \mf{Re} Z(A)$. 
Recall that the pair $(\beta, \omega)$ is in $\mca V(X)$ by $\omega ^2 >2$. 
Thus we have $\mf{Re} Z(A)>0$.  
We have proved the assertion. \vspace{5pt} \\
\textit{The Proof of (2). }

By the assumption, we have $\frac{n_E}{r_E} < \frac{n_A}{r_A}< x$ and $r_A n_E - r_E n_A < 0$. 
In addition, both $\lambda _E$ and $\lambda _A $ are negative. 
Similarly to the case $(1)$, we have 
\begin{eqnarray*}
\arg Z(E) < \arg Z(A)	&\iff & \frac{\mf{Re} Z(E)}{\lambda _E} < \frac{\mf{Re}Z (A)}{\lambda _A} \\
						&\iff & 0 > N_{A,E}(x,y).
\end{eqnarray*}
We have the same formula as (\ref{N}) for $N_{A,E}(x,y)$ with two differences. 
One is $(r_A n_E - r_E n_A) < 0$ (this is obvious). The other is $\mf{a} \leq \frac{n_A}{r_A}$.  
So we shall prove the second inequality $\mf{a} \leq \frac{n_A}{r_A}$. 
In fact
\begin{eqnarray}
\frac{n_A}{r_A} \geq \mf{a}	&\iff & \frac{n_A}{r_A}- \frac{n_E}{r_E} \geq \frac{r_E}{d r_A (r_E n_A -r_A n_E)} \notag \\
							&\iff & \frac{(r_E n_A -r_A n_E)^2}{r_E^2} \geq \frac{1}{d}. \label{Last} 		
\end{eqnarray}
The inequality (\ref{Last}) holds by $\sqrt{d} \geq r_E$. 

Since $r_A n_E -r_E n_A$ is negative, $N_{A,E}(x,y)$ is strict decreasing to $y> 1/\sqrt{d}$. 
Similarly to $(1)$, since the inequality $\mf{a} \leq \frac{n_A}{r_A}$ holds, $N_{A,E}(x,y)$ is strict decreasing with respect to $x > \frac{n_A}{r_A}$. 
Thus we have $N_{A,E}(x,y) < N_{A,E}(\frac{n_A}{r_A}, \frac{1}{\sqrt{d}})$. 
Hence it is enough to show $N_{A,E}(\frac{n_A}{r_A},y) < 0$. 
This follows from $\omega ^2 >2$. 
So we have proved the assertion $(2)$. 
\end{proof}

Now we are ready to prove the main theorem of this section. 

\begin{thm}\label{4.7}
Let $(X,L)$ be a generic K3 with $\deg X =2d$, 
$\sigma _{(\beta,\omega)}$ in $V(X)_{>2}$ and 
$E$ a torsion free sheaf with $v(E)^2=0$ and $\rank E \leq \sqrt d$. 

$(1)$ Assume that $E$ is Gieseker stable and $\beta \omega < \mu_{\omega}(E)$. 
Then $E$ is $\sigma_{(\beta,\omega)}$-stable. 

$(2)$ Assume that $E$ is $\mu$-stable locally free and $\mu _{\omega }(E) \leq \beta  \omega$. 
Then $E$ is $\sigma _{(\beta,\omega)}$-stable.  
\end{thm}

\begin{proof}
We put $\sigma _{(\beta,\omega)} = (Z,\mca P)$. 
The assumption of $(1)$ implies $E \in \mca P((0,1])$ and 
that of $(2)$ implies $E \in \mca P((-1,0])$. 
\vspace{5pt}

\textit{Proof of (1)}. 
Suppose to the contrary that $E$ is not $\sigma _{(\beta,\omega)}$-stable. By Proposition \ref{4.2}, 
there is a $\sigma _{(\beta,\omega)}$-stable sheaf $S$ with $v(S)^2 =-2$, $\mu_{\omega}(S) < \mu _{\omega}(E)$ 
and  $\arg Z(S) \geq \arg Z(E)$. 
This contradicts Lemma \ref{4.6} (1). Hence $E$ is $\sigma _{(\beta,\omega)}$-stable. 
\vspace{5pt} 

\textit{Proof of (2)}. 
Suppose to the contrary that $E$ is not $\sigma _{(\beta,\omega)}$-stable. 
Then by Lemma \ref{4.5}, 
there is a $\sigma _{(\beta,\omega)}$-stable sheaf $S$ with $\mu_{\omega}(E) < \mu_{\omega}(S)$, $v(S)^2 =-2$ 
and $\arg Z(S) \leq \arg Z(E)$. 
This contradicts Lemma \ref{4.6} (2). 
Hence $E$ is $\sigma _{(\beta,\omega)}$-stable. 
\end{proof}

\begin{cor}\label{4.8}
Let $(X,L)$ be a generic K3 with $\deg X =2d$ and 
let $E$ be a $\mu$-stable locally free sheaf with $\rank E \leq \sqrt d$. 
Then for all $\sigma  \in U(X)_{>2}$, $E$ is $\sigma $-stable. 
\end{cor}

\begin{proof}
Let $\sigma \in U(X)$ and $\tilde g \in \tilde{GL}^+( 2,\bb R)$. 
$E$ is $\sigma$-stable if and only if $E$ is $\sigma \cdot \tilde g$-stable. 
Thus we have finished the proof by Theorem \ref{4.7}. 
\end{proof}

The assumption $\rank E \leq \sqrt d$ may seem to be artificial but it is just the same as the condition $r \leq s$ in Theorem \ref{3.3}. 
In Example \ref{5.3} we shall show that the assumption is optimal.

\section{$\sigma$-stability of spherical sheaves}


Let the notations be as in Section $4$. 
In this section, for a generic K3 $(X,L)$, 
we prove that some spherical sheaves are $\sigma$-stable for all $\sigma \in U(X)_{>2}$. 
We start in this section with a brief review of spherical objects. 
An object $S \in D(X)$ is called a \textit{spherical object}
\footnote{This definition is ``K3'' version. 
More generalized definition of spherical object appears in \cite[Chapter 8]{Huy} or \cite{ST}. } 
if the morphism space $\Hom_X^i(S,S)$ is
\[
\Hom ^i_X(S,S) = \begin{cases}\bb C & (i=0,2) \\ 0 & (\mbox{otherwise}).  \end{cases} 
\] 
By virtue of \cite{ST}, we can define an autoequivalence $T_S$ called a \textit{spherical twist}. 
For $E \in D(X)$ the complex $T_S(E)$ is isomorphic to
\begin{equation}
T_S(E) \simeq \mbox{the mapping cone of }\Big( \Hom _X(S, E[*]) \otimes S \stackrel{\mathit{ev}}{\to} E  \Big) , \label{5}
\end{equation}
where $\mathit{ev}$ is the evaluation map. 

In general it is difficult to compute $T_S(E)$, but much easier to compute the Mukai vector $v(T_S(E))$. 
In fact, we have 
\begin{equation}
v(T_S(E)) = v(E) + \< v(E) , v(S)  \> v(S). \label{twist}
\end{equation}
Recall that any equivalence $\Phi :D(Y) \to D(X)$ induces an isometry $\Phi ^H : \mca N(Y) \to \mca N(X)$. 
Since $v(\Phi (E)) = \Phi ^H (v(E) )$, we have $T_S^H \circ T_S^H = id _{\mca N(X)}$ by (\ref{twist}). 


\begin{ex}\label{5.5}
Let $X$ be a projective K3 surface. Then any line bundle $M$ is spherical. 
The spherical twist $T_M(\mca O_x)$ of $ \mca O_x$ by $M$ is $\mca I_x \otimes M [1]$ where $\mca I_x$ is the ideal sheaf of the closed point $x \in X$. 
This follows from the formula (\ref{5})
\end{ex}

\begin{prop}\label{Mukai}
Let $(X,L)$ be a generic K3 and $S$ a spherical sheaf.  
Then $S$ is a $\mu$-stable locally free sheaf. 
\end{prop}

\begin{proof}
We first show that $S$ is locally free. 
Let $t(S)$ be the maximal torsion subsheaf of $S$. 
Then we have the following exact sequence of sheaves:
\[
\begin{CD}
0 @>>> t(S) @>>> S @>>> S/ t(S) @>>> 0 .
\end{CD}
\]
Since $\Hom (t(S), S/ t(S))=0$, the result \cite[Corollary 2.8]{Muk} gives us the following inequality:
\[
0 \leq \hom^1(t(S),t(S)) + \hom^1(S/t(S), S/t(S)) \leq \hom^1(S,S) =0. 
\]
Thus $v(t(S))^2 <0$ unless $t(S)=0$. 
However $v(t(S))^2 \geq 0$ for $t(S)$ is torsion and $S$ is of Picard number $1$. 
Hence $t(S)=0$. 
Thus $S$ is torsion free. 
Then the local-freeness of $S$ comes from \cite[Proposition 3.3]{Muk}. 

Finally we show that $S$ is $\mu$-stable. 
Since $v(S)^2=-2$, the greatest common divisor of $(r_S,n_S)$ is $1$. 
Then the $\mu$-stability of $S$ follows from \cite[Lemma 1.2.14]{HL} under the assumption that the Picard number is one. 
\end{proof}

The following lemma is a modified version of Lemma \ref{4.6}. 

\begin{lem}\label{5.1}
Let $(X,L)$ be a generic K3 with $\deg X = 2d$, $\sigma_{(\beta,\omega)} \in V(X)_{>2}$ 
and both $A $ and $E$ spherical sheaves 
with $\rank E \leq \sqrt{d}$. 

$(1)$ Assume that $\beta \omega  < \mu_{\omega}(A) < \mu_{\omega}(E)$. Then $0 < \arg Z(A) < \arg Z(E) < 1$.  

$(2)$ Assume that $\mu_{\omega}(E)< \mu _{\omega}(A) < \beta  \omega$. Then $-1 < \arg Z(E) < \arg Z(A) < 0$. 
\end{lem}

\begin{proof}
Since $\mr{NS}(X) = \bb Z \cdot L$, we can put 
\[
\beta = x L , \omega =y L, v(E)= r_E \+ n_E L \+ s_E, \mbox{ and }
v(A) = r_A \+ n_A L \+ s_A .
\]
Then, by the formula (\ref{4.0}) in Section 4, we have 
\begin{eqnarray*}
	& Z(E) = -\frac{1}{r_E} + d r_E(y^2 - \frac{\lambda _E^2}{r_E^2}) + 2 \sqrt{-1}d y \lambda _E & \mbox{and} \\
	& Z(A) = -\frac{1}{r_A} + d r_A(y^2 - \frac{\lambda _A^2}{r_A^2}) + 2 \sqrt{-1}d y \lambda _A, & 
\end{eqnarray*}
where $\lambda _E = n_E -r_E x$ and $\lambda _A = n_A -r_A x$. 

We only prove $(1)$, 
because the proof of $(2)$ is essentially the same as not only the proof of $(1)$ but also it of Lemma \ref{4.6}. 

Since both $\lambda _A$ and $\lambda _E$ are positive by the assumption, we know that 
\[
\arg Z(A) < \arg Z(E) \iff N_{A,E}(x,y) > 0.
\]
Similarly to Lemma \ref{4.6}, we have 
\begin{eqnarray*}
N_{A,E}(x,y)	&=&	dy^2(r_A \lambda _E -r_E \lambda _A) + d \lambda _A \lambda _E 
					\Big( \frac{\lambda _E}{r_E} -\frac{\lambda _A}{r_A} \Big) + 
					\frac{\lambda _A}{r_E} -\frac{\lambda _E}{r_A} \\
				&=&d(r_A n_E -r_E n_A)y^2 + d(r_A n_E -r_E n_A)( x- \mf{a} )^2 \\
				& &+  (\mbox{other terms}),  
\end{eqnarray*}
where $\mf{a}$ is 
\[
\mf {a}=\frac{1}{2}\Big( \frac{n_A}{r_A} + \frac{n_E}{r_E} + \frac{1}{d(r_An_E- r_E n_A)}
					\Big( \frac{r_A}{r_E}- \frac{r_E}{r_A} \Big)  \Big).
\]

Then we shall show that $\frac{n_A}{r_A} < \mf{a}$. 
Since the integer $r_E n_A - r_A n_E$ is negative, we have 
\begin{eqnarray}
\frac{n_A}{r_A} < \mf {a}	&\iff &	\frac{n_A}{r_A} - \frac{n_E}{r_E} < 
										\frac{1}{d(r_A n_E -r_E n_A)}\Big(\frac{r_A}{r_E}-\frac{r_E}{r_A}  \Big) \notag \\
								&\iff &	{(r_E n_A - r_A n_E)^2} > \frac{r_E^2 -r_A ^2}{d} .\label{tofu}
\end{eqnarray}
By the assumption $0 < \rank E \leq \sqrt{d}$ we have $(r_E n_A - r_A n_E)^2 \geq \frac{r_E^2}{d}$. 
Thus the last inequality (\ref{tofu}) holds. 

Since $\frac{n_A}{r_A} < \mf{a}$, $N_{A,E}(x,y)$ is strict decreasing with respect to $x < \frac{n_A}{r_A}$. 
Moreover by $r_A n_E -r_E n_A >0$, $N_{A,E}(x.y)$ is strict increasing with respect to $y > \frac{1}{\sqrt{d}}$. 
Thus we have $N_{A,E}(x,y) > N_{A,E}(\frac{n_A}{r_A}, \frac{1}{\sqrt{d}})$. 
Thus it is enough to show that $N_{A,E}(\frac{n_A}{r_A},y)>0$. 
This follows from $\omega^2 >2$. 
Hence we have $N_{A,E}(x,y)>0$ for all $(\beta,\omega)$ satisfying the assumption. 
\end{proof}

In the same way as Theorem \ref{4.7}, we have the following proposition. 

\begin{prop}\label{5.2}
Let $(X,L)$ be a generic K3 with $\deg X =2d$ and $E$ a spherical sheaf on $X$ with $\rank  E \leq \sqrt{d}$. 
Then $E$ is $\sigma$-stable for all $\sigma \in U(X)_{>2}$.  
\end{prop}

The proof is essentially the same as that of Theorem \ref{4.7}. 

\begin{proof}
We can assume that $\sigma = \sigma _{(\beta,\omega)} =(Z ,\mca P) \in V(X)_{>2}$. 
Since $E$ is $\mu$-stable by Proposition \ref{Mukai}, $ E \in \mca P((0,1])$ or $E \in \mca P((-1,0])$. 

Let $E \in \mca P((0,1])$. 
Assume to the contrary that $E$ is not $\sigma $-stable. 
From Proposition \ref{4.2}
we know that there is a $\sigma $-stable torsion free sheaf $S \in \mca P((0,1])$ with 
$v(S)^2= -2$, $\mu_{\omega}(S) < \mu_{\omega}(E)$ and $\arg Z(E) \leq \arg Z(S)$. 
However, by Lemma \ref{5.1}, we have $\arg Z(S) < \arg Z(E)$. 
This is contradiction.

Let $E \in \mca P((-1,0])$. 
Assume to the contrary that 
$E$ is not $\sigma$-stable. 
Then, by Proposition \ref{4.5}, 
there is a $\sigma$-stable sheaf $S'$ with $\mu_{\omega}(E) < \mu_{\omega}(S')$, $v(S')^2=-2$ and 
$\arg Z(S') \leq  \arg Z(E)$. 
However, by Lemma \ref{5.1}, we have $\arg Z(S') > \arg Z(E)$. 
So $E$ is $\sigma $-stable. 
\end{proof}


In Example \ref{5.3}, we show that the assumption on the rank of $E$ in Theorem \ref{4.7} is optimal. Namely we give an example of a Gieseker stable sheaf $E$ with $\rank E > \sqrt{d}$ which is not $\sigma$-stable for some $\sigma \in V(X)_{>2}$.

\begin{ex}\label{5.3}
Let $(X,L)$ be a generic K3 with $\deg X=2d$, and 
$E$ a Gieseker stable locally free sheaf with $\<v(E)\> ^2= \< r_E \+ L \+ s_E \>^2=0$ where $v(E) = r_E \+ L \+ s_E$ with $r_E > \sqrt{d}$. 
Then we claim that there is a $\sigma  \in V(X)_{>2}$ such that $E$ is not $\sigma$-semistable. 
To prove our claim, it is enough to find $\sigma _{(\beta,\omega)}= (Z,\mca P) \in V(X)_{>2}$ such that 
\begin{equation}
\arg Z(\mca O_X) > \arg Z (E). \label{katei} 
\end{equation} 

In fact, assume that such a stability condition $\sigma_0 \in V(X)_{>2}$ exists. 
By Lemma \ref{5.4} (below), we have $\chi(\mca O_X,E)>0$. 
Since $\mu _{\omega}(\mca O_X) < \mu_{\omega}(E)$, 
$\Hom_X^2(\mca O_X, E)^* =  \Hom _{\mca O_X}(E,\mca O_X)=0$. 
Thus we have 
\begin{equation}
0 < \chi(\mca O_X, E) = \hom ^0(\mca O_X,E) - \hom ^1(\mca O_X,E) \leq \hom ^0(\mca O_X,E). 
\label{kainoatai}
\end{equation}
Recall that $\mca O_X$ is $\sigma _0$-stable by Proposition \ref{5.2}. 
If $E$ is $\sigma_0$-semistable, we have $\Hom _X(\mca O_X, E) = 0$ by the assumption (\ref{katei}). 
This contradicts (\ref{kainoatai}). 
Hence $E$ is not $\sigma _0$-semistable. 

We finally show that there is a $\sigma _{(\beta,\omega)}\in V(X)_{>2}$ satisfying the condition (\ref{katei}). 
We put $(\beta,\omega)=(x L, y L)$. 
Let $N_{A,E}(x,y)$ be the function defined by (\ref{NAE}). 
Since $v(\mca O_X)=1 \+ 0 \+ 1$ and $v(E) = r_E \+ L \+ s_E$, we have 
\[
N_{\mca O_X,E}(x,y) = dx^2 + \Big( r_E -\frac{d}{r_E} \Big)x + dy^2-1.
\]
Take $x<0$. Then the condition (\ref{katei}) is equivalent to 
\[
N_{\mca O_X,E}(x,y)<0 .
\]
Let us consider the special case $dy^2=1$. 
This means $\omega ^2=2$. 
If $dy^2=1$, the solutions of $N_{\mca O_X,E}(x,\sqrt{1/d})=0$ are 
\[
x=0, \alpha \mbox{, where } \alpha =\frac{d-r_E ^2}{r_E d}.
\]
The region defined by $N_{\mca O_X,E}(x,y) < 0$ is the inside of the following circle:
\begin{center}
\includegraphics[height=40mm]{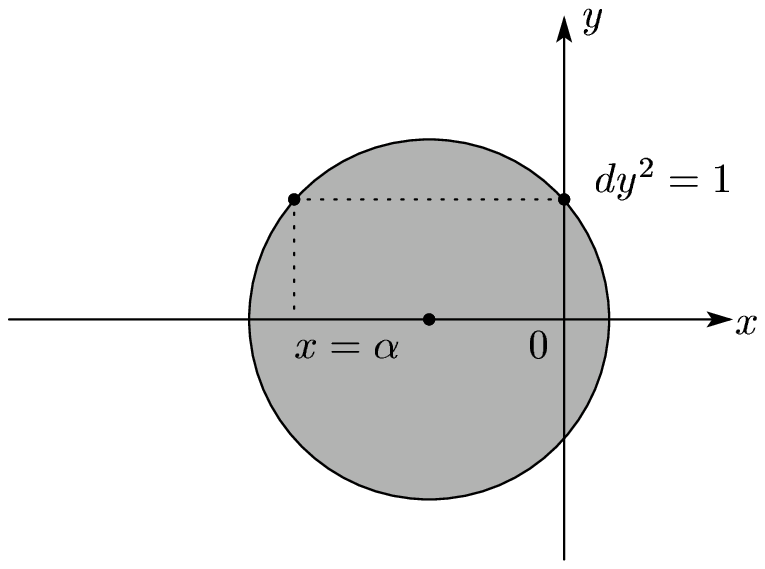}
\end{center}
Hence we can choose $\sigma_{(\beta,\omega)} \in V(X)_{>2}$ so that 
$x < 0$ and $N_{\mca O_X,E}(x,y) < 0$. 
\end{ex}

\begin{lem}\label{5.4}
Let $(X,L)$ be a generic K3, let $E $ be a sheaf with $v(E)^2 \leq 0$ and $\rank E >0$, 
and let $A$ be a sheaf with $v(A)^2 <0 $. 
Then we have $\chi( A,E) > 0$. 
\end{lem}

\begin{proof}
We put 
\[
v(A) = r_A \+ n_A L \+ s_A \mbox{, and } v(E) = r_E \+ n_E L \+ s_E .
\]
Since $v(A) ^2< 0$ and the Picard number is one, $r_A$ should be positive. 
So we have
\[
\frac{s_A}{r_A}= \frac{L^2}{2}\Big( \frac{n_A}{r_A} \Big)^2 - \frac{v(A)^2}{2r_A^2}
\mbox{ and }
\frac{s_E}{r_E}= \frac{L^2}{2}\Big( \frac{n_E}{r_E} \Big)^2 - \frac{v(E)^2}{2r_E^2} .
\]
Then
\begin{eqnarray*}
\frac{\chi (A, E)}{r_A r_E}	&=&	\frac{-\< v(A),v(E) \>}{r_A r_E} \\
							&=& \frac{L^2}{2}\Big( \frac{n_A}{r_A} - \frac{n_E}{r_E} \Big)^2 - ( \frac{v(E)^2}{2 r_E^2} + \frac{v(A)^2}{2 r_A^2} )\\
							&>& 0. 
\end{eqnarray*}
Hence $\chi(A,E)>0$. 
\end{proof}

By virtue of Proposition \ref{5.2} 
we can determine the HN filtrations of some special complexes for $\sigma \in V(X)_{>2}$. 
We remark that there is a similar assertion to the following two corollaries in \cite[Proposition 2.15]{HMS} when $X$ is a K3 surface with $\mr{NS}(X)=0$.

\begin{cor}\label{5.6}
Let $(X,L)$ be a generic K3 with $\deg X =2d$, $\sigma = \sigma _{(\beta,\omega)} = (Z,\mca P)$ in $V(X)_{>2}$ and 
$S$ a spherical sheaf on $X$ with $\rank S \leq \sqrt{d}$. 
We put $\beta= bL$ and $v(S) = r \+ nL \+ s$. 

$(1)$ If $b > \frac{n}{r}$, then $T_S(\mca O_x)$ is not $\sigma $-semistable. The HN filtration of $T_S(\mca O_x)$ is given by 
\begin{equation}
\xymatrix{
0\ar[rr]	&	&\mca O_x \ar[ld]\ar[rr]	& 	& T_S(\mca O_x)\ar[ld] \\
	&\mca O_x \ar@{-->}[ul]^{[1]}& 		& S^{\+ r}[1]\ar@{-->}[ul]^{[1]}&  
}.
\label{HN}
\end{equation}

$(2)$ If $b=\frac{n}{r}$, then $T_S(\mca O_x)$ is $\sigma $-semistable. The JH filtration of $T_S(\mca O_x)$ is given by the sequence $(\ref{HN})$. 

$(3)$ If $b < \frac{n}{r}$ and $r  \leq  d^{\frac{1}{4}}$, then $T_S(\mca O_x)$ is $\sigma $-stable. 
\end{cor}

\begin{proof}
We first remark that the sequence of distinguished triangles (\ref{HN}) comes from the formula (\ref{5}). 

(1) Assume that $b > \frac{n}{r}$. Then $S^{\+ r}$ is in $\mca P((-1,0])$ and it is $\sigma$-semistable by Proposition \ref{5.2}. 
Hence $\arg Z(\mca O_x) > \arg Z(S^{\+ r} [1]) > 0$. 
Thus the sequence (\ref{HN}) is the HN filtration of $T_S(\mca O_x)$. 

(2) If $b =\frac{n}{r}$ then $\arg Z(\mca O_x) = \arg Z(S^{\+ r} [1])$. 
By Proposition \ref{5.2}, $S$ is $\sigma$-stable. 
Thus (\ref{HN}) is a JH filtration of $T_S(\mca O_x)$. 

(3) We put $\tilde S_x = \Ker (S ^{\+ r} \to \mca O_x)$. 
Note that $\rank \tilde S_x = r^2$. 
Then $T_S(\mca O_x) = \tilde S_x [1]$. 
So it is enough to show that $\tilde S_x$ is $\sigma$-stable. 
Since $T_S$ is an equivalence we have 
\[
\hom_X^0(\tilde S_x , \tilde S_x) = 1 \mbox{, } \hom_X^1(\tilde S_x,\tilde S_x) = 2 \mbox{ and $v(\tilde S_x)$ is primitive}.  
\]
Thus $\tilde S_x$ is Gieseker stable by \cite[Proposition 3.14]{Muk}. 
Then $\tilde S_x$ is $\sigma$-stable by Theorem \ref{4.7} (1)
\end{proof}

By Corollary \ref{5.6} (1), we can see that 
it is impossible to remove the assumption of local-freeness in Theorem \ref{4.7} (2). 

\begin{cor}\label{5.7}
Let the notations be as in Corollary \ref{5.6}. 

$(1)$ If $b \leq \frac{n}{r}$ and $r \leq d^{\frac{1}{4}}$ then the HN filtration of $T_S^n(\mca O_x)$ $(n >1 )$ is given by
\begin{center}
\resizebox{0.99\hsize}{!}{
$
\xymatrix{
0\ar[rr] && T_S(\mca O_x)\ar[ld]\ar[rr]	&	&	T_S^2(\mca O_x)\ar[ld]\ar[r] & \cdots \ar[r] & T_S^{n-1}(\mca O_x) \ar[rr]	&	  &T_S^n(\mca O_x)\ar[ld] \\ 
&T_S(\mca O_x)\ar@{-->}[lu]^{[1]}&	& S^{\+ r} \ar@{-->}[ul]^{[1]}& 		& 		  &			&S^{\+ r}[2-n]\ar@{-->}[ul]^{[1]}& 
}.
$
}
\end{center}

$(2)$ If $b > \frac{n}{r}$, then the HN filtration of $T_S^n(\mca O_x)$ is 
\begin{center}
\resizebox{0.99\hsize}{!}{
$
\xymatrix{
0\ar[rr] && \mca O_x\ar[ld] \ar[rr]	&	&	T_S(\mca O_x)\ar[ld]\ar[r] & \cdots \ar[r] & T_S^{n-1}(\mca O_x)\ar[rr]	&	  &T_S^n(\mca O_x)\ar[ld]\\ 
&\mca O_x\ar@{-->}[lu]^{[1]}&	& S^{\+ r}[1] \ar@{-->}[ul]^{[1]}& 		& 		  &			&S^{\+ r}[2-n]\ar@{-->}[ul]^{[1]}& 
}.
$
}
\end{center}
\end{cor}

\begin{proof}
By (\ref{5}), we obtain the following distinguished triangle:
\[
\begin{CD}
S^{\+ r} @>>> \mca O_x @>>> T_S(\mca O_x) @>>> S^{\+ r}[1]
\end{CD}.
\] 
Since $T_S(S)\simeq S[-1]$
\footnote{One can prove this fact $T_S(S) \simeq S[-1]$ easily in the following way. We have the natural exact sequence of sheaves by taking cohomologies of the distinguished triangle arising from (\ref{5}). Then the fact follows from the exact sequence of sheaves. See also \cite[Exercise 8.5]{Huy}. }, 
we can easily show that the two sequences of triangles exist. 
By Corollary \ref{5.6}, both sequences are the HN filtrations of $T_S^n(\mca O_x)$. 
\end{proof}

\section{Applications of Theorem \ref{thm2}}

In this section we deal with two applications of Theorem \ref{thm2}. 
We first observe the morphism $\Phi _*$ between the space of stability conditions induced by an equivalence $\Phi$ of triangulated categories. 

Let $X$ and $Y$ be projective K3 surfaces, and $\Phi:D(Y) \to D(X)$ an equivalence. 
Then $\Phi$ induces a natural morphism $\Phi_* : \Stab(Y) \to \Stab(X)$ as follows:
\[
\begin{matrix}
	&	\Phi_* : \Stab(Y) \to  \Stab (X),\  	
		\Phi _*\bigl( ( Z_Y,\mca P_Y) \bigr)= (Z_X,\mca P_X)	& \\
	& \mbox{where }Z_X(E) = Z_Y \bigl( \Phi ^{-1}(E) \bigr) ,\ 
	\mbox{and }\mca P_X(\phi )= \Phi \Bigl( \mca P_Y(\phi ) \Bigr) .
\end{matrix}
\]

Then the following proposition is almost obvious. 

\begin{prop}\label{6.1}
Let $X$ and $Y$ be projective K3 surfaces, and $\Phi:D(Y) \to D(X)$ an equivalence. 
For $\sigma \in U(X)$, $\sigma$ is in $\Phi _* (U(Y) )$ if and only if $\Phi (\mca O_y) $ is $\sigma$-stable with the same phase 
for all closed points $y \in Y$. 
\end{prop}

\begin{proof}
By the definition of $\Phi _* : \Stab (Y) \to \Stab (X)$, $\Phi _*(U(Y))$ is given by:
\begin{eqnarray*}
\Phi _*(U(Y))	&=&	\Phi _* \bigr( \{ \sigma \in \Stab (Y) | \sigma \mbox{ is good, }\mca O_y \mbox{ is $\sigma$-stable }    
			(\forall y \in Y)\} \bigl) \\
				&=& \{ \tau \in \Stab (X) | \tau \mbox{ is good, }\Phi (\mca O_y) \mbox{ is $\tau$-stable }    
			(\forall y \in Y)\} .
\end{eqnarray*}
Recall that the $\Phi $ induces the isometry $\Phi ^H :\mca N(Y) \to \mca N(X)$. 
So if $\sigma \in \Stab (Y)$ is good, then $\Phi_* (\sigma)$ is also good. 
This completes the proof. 
\end{proof}

Let us consider the first application of Theorem \ref{4.7}.

\begin{ex}\label{6.2}
In this example we claim that there is a pair $(E,\tau)$ such that a true complex $E \in D(X)$ is $\tau$-stable for $\tau \in V(X)\backslash V(X)_{>2}$.

We first define a special subset $D^M$ of $V(X)\backslash V(X)_{>2}$ depending on a line bundle $M$ in the following way. 
We put
$V(X)_{>2}^M$ for $M$ by  
\[
V(X)_{>2}^M := \{ \sigma _{(\beta,\omega)}  \in V(X)_{>2} | \beta \omega  < \mu_{\omega}(M)  \} . 
\]
By Proposition \ref{6.1} and Corollary \ref{5.6} (3), we see $V(X)_{>2}^M \subset (T_M)_*(U(X)) \cap V(X)$. 
We also put $U(X)_{>2}^M := V(X)_{>2}^M \cdot \tilde{GL}^+(2,\bb R)$.
By Remark \ref{r1.7}, we see $U(X)_{>2}^M \subset (T_M)_*(U(X)) \cap U(X)$. 
Then we define 
\[
D^M := T_{M*}^{-1}\big( U(X)_{>2}^M \big) \cap V(X).
\] 
Since $T_M = (\otimes M) \circ T_{\mca O_X} \circ (\otimes M^{-1})$ 
we see  that $D^{M}$ is the following half circle:
\begin{center}
\includegraphics[clip,height=45mm]{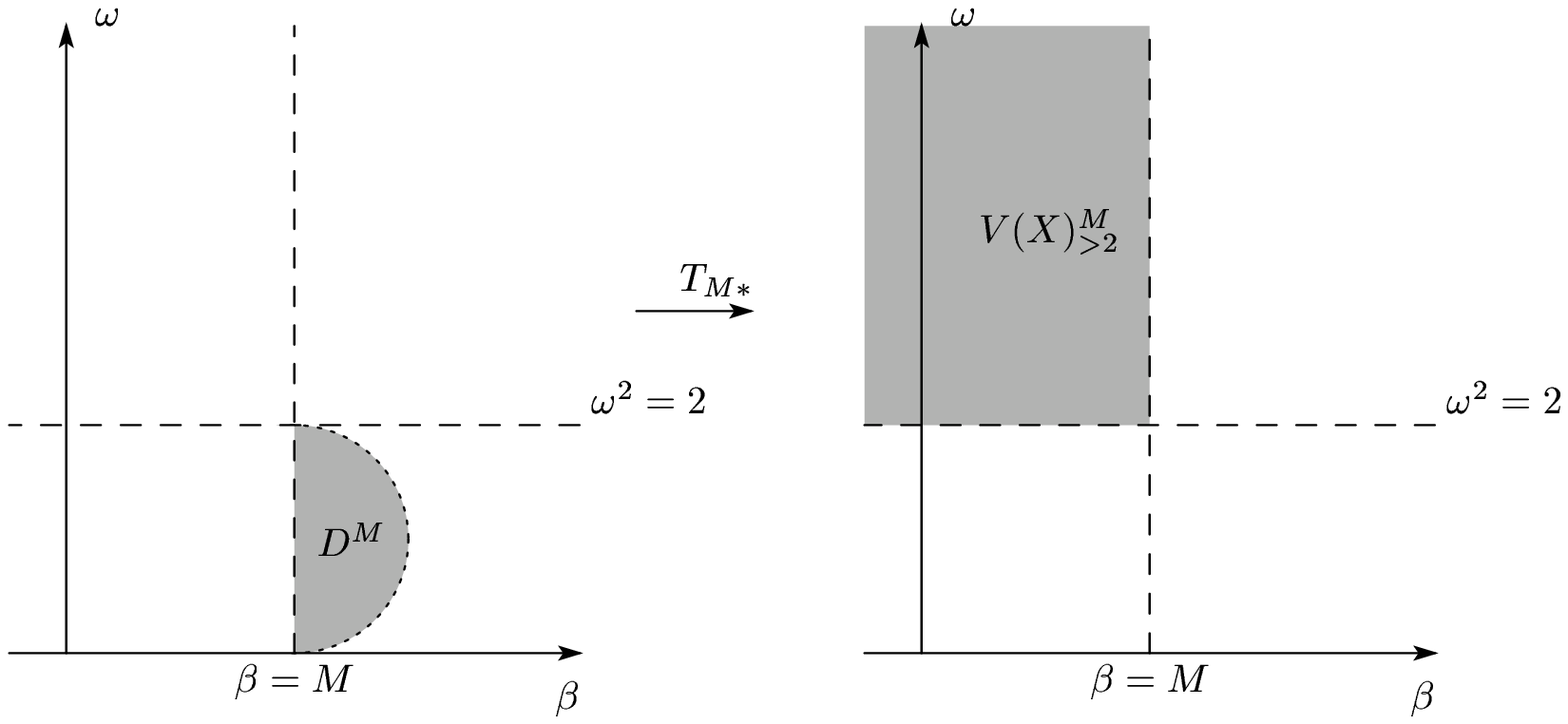}
\end{center}
Thus $D^M \subset V(X)\backslash V(X)_{>2}$. 
 
Next we show that there is a true complex $E \in D(X)$ which is $\tau$-stable 
for $\tau \in D^M$.  
In fact, by Proposition \ref{6.1}, 
$E \in D(X)$ is $\sigma$-stable for any $\sigma \in V(X)_{>2}^M$ 
(for example $E$ is a torsion free sheaf in Theorem \ref{thm2} or $\mca O_x$), 
if and only if $T_M^{-1}(E)$ is $\tau$-stable for any $\tau \in D^M$. 
For instance, $T_M^{-1}(\mca O_x )$ is truly complex which is $\tau $-stable for any $\tau \in D^M$.  
By the definition of $T_M$, we can easily compute the $i$-th cohomology $H^i$ of $T_M^{-1}(\mca O_x)$. 
In fact we have 
\[
H^i = 	\begin{cases}	\mca O_x & (i=0) \\
						M & (i=-1) \\
						0 & (\mbox{otherwise}) .
		\end{cases} 
\]
\end{ex}

The crucial part of Example \ref{6.2} is that 
the spherical twist $T_M$ enables us to exchange the unbounded region $V(X)_{>2}^M$ into the bounded region $D^M$. 
We use this idea in the proof of Theorem \ref{thm1}. 
\vspace{5pt}

Next we shall explain the second application. 
In general spherical twists send sheaves to complexes. 
We first show this easy statement in a special case. 

\begin{lem}\label{6.3}
Let $(X,L)$ be a generic K3, and $E$ a Gieseker stable torsion free sheaf with $v(E)^2 \leq 0$. 
Then there is a line bundle $M$ such that the spherical twist $T_M(E)$ of $E$ is a true complex with 
$r' \neq 0$ where $v(T_M(E)) = r ' \+ \Delta ' \+ s'$. 
\end{lem}

\begin{proof}
Let $v(E) = r_E \+ n_E L \+ s_E$ and let $M = m L$ be a line bundle with 
\begin{equation}
\frac{n_E }{r_E  }< m  . \label{open}
\end{equation}
Here we compute $v(T_M(E))$: 
\begin{eqnarray}
v(T_M(E))	&=& v(E) + \< v(E), v(M) \> v(M) \notag \\
			&=& r' \+ n' L \+ s' . \notag \label{closed}
\end{eqnarray}
The condition $r'=0$ is a closed condition and the condition (\ref{open}) is open. 
Hence we can choose $M$ so that $r' \neq 0$ and $M$ satisfies the condition (\ref{open})

Let $H^i$ be the $i$-th cohomology of $T_M(E)$. 
By the definition of spherical twists, we obtain the following exact sequence of sheaves:
\[
\begin{CD}
0 @>>> \Hom ^0_X(M,E)\otimes M	@>>> E @>>> H^0  \\
@>>>\Hom ^1_X(M,E)\otimes M	@>>> 0 @>>> H^1  \\
@>>>\Hom ^2_X(M,E)\otimes M	@>>> 0 @>>> H^2 @>>> 0 
\end{CD}
\]
Since both $M$ and $E$ are Gieseker stable, $\Hom^0_X(M,E)=0$ by (\ref{open}). 
Hence $H^0$ is not $0$. 
By Lemma \ref{5.4}, we have $\Hom^2_X(M,E) \neq 0$. So $H^1 \neq 0$. 
Thus $T_M(E)$ is a complex. 
\end{proof}

The following lemma is due to \cite{Bri2} and \cite{Tod}. 

\begin{lem}\label{6.4}(\cite[Proposition 14.2]{Bri2}, \cite[Proposition 6.4]{Tod})
Let $X$ be a projective K3 surface, $\sigma_{(\beta,\omega)}=(Z,\mca P) \in V(X)$ and $E$ in $\mca P((0,1])$.  
We put $v(E)=r \+ \Delta  \+ s$. 

$(1)$ Assume that $r>0$. 
If $E$ is $\sigma_{(\beta,n \omega)}$-semistable for any sufficiently large $n \gg 0$, 
then $E$ is a torsion free sheaf. 

$(2)$ Assume that $r=0$. 
If $E$ is $\sigma_{(\beta,n \omega)}$-semistable for any sufficiently large $n \gg 0$, 
then $E$ is a torsion sheaf. 
\end{lem}

The first assertion of Lemma \ref{6.4} are proved by \cite{Bri2} and the second one proved by \cite{Tod}. We can prove the second assertion in a similar way to \cite{Bri2}.  

In the next proposition, we show that it is impossible to extend Theorem \ref{thm2} to $V(X)$ by using Lemma \ref{6.4} and the idea of Example \ref{6.2}.

\begin{prop}\label{6.5}
Let $(X,L)$ be a generic K3 and $E$ a Gieseker stable torsion free sheaf with $v(E)^2 \leq 0$. 
Then there is a $\sigma $ in $V(X)$ such that $E$ is not $\sigma$-semistable. 
\end{prop}

\begin{proof}
Assume that $E$ is $\sigma $-semistable for all $\sigma \in V(X)$. 
By Lemma \ref{6.3}, there is a line bundle $M $ such that $T_M(E)$ is a complex with 
$r'\neq 0$ where $v(T_M(E)) = r' \+ \Delta ' \+ s' $. 
By a shift of $T_M(E)$ we can assume that $r' >0$ if necessary.  
By the assumption $T_M(E)$ is $\sigma $-semistable for all $\sigma $ not only in $(T_M)_* V(X)$ but also in $(T_M)_* U(X)$. 

Recall that, $(T_M)_*(U(X)) \cap V(X)$ contains the set $V(X)_{>2}^M$ defined in Example \ref{6.2}. 
Hence, 
there is a $\tau_{(\beta,\omega)}=(Z,\mca P) \in V(X) _{>2}^M$ such that 
\[
\beta \omega < \frac{\Delta '}{r'}\omega .
\]
This implies that $T_{M}(E)[2n] $ is in $\mca P((0,1])$ for some $n \in \bb Z$. 
By Lemma \ref{6.4} (1), $T_M(E)[2n]$ should be a sheaf. This contradicts the fact that $T_M(E)$ is a true complex. 
\end{proof}

\begin{thm}\label{6.6}
Let $(X,L) $ be a generic K3 and $E \in D(X)$. 
We assume that $\Hom_X^0(E,E) = \bb C$, $v(E)$ is primitive and $v(E)^2=0$. 
If $E$ is $\sigma $-semistable for all $\sigma \in V(X)$, then $E$ is $\mca O_x$ for some $x \in X$ up to shifts. 
\end{thm}

\begin{proof}
We put $v(E)=r_E \+ n_E L \+ s_E$. 

Assume that $r_E \neq 0$. If $E$ is $\sigma $-semistable, then $E[1]$ is also $\sigma$-semistable. 
Thus we can assume that $r_E >0$. Let $\phi$ be the phase of $E$. 
Then we can assume $\phi \in (-1,1]$ by even shifts.  
There is an $\bb R$ divisor $\beta = bL$ such that $b < n_E /r_E$. 
Let us consider $\sigma _{(\beta,\omega)}=(Z,\mca P)$ for all ample divisors $\omega $ with $\omega ^2>2$.  
Notice that $E$ is in $\mca P((0,1])$. 
By Lemma \ref{6.4}, $E$ should be a torsion free sheaf. 
In addition, $E$ is a Gieseker stable sheaf by \cite[Proposition 3.14]{Muk}. 
This contradicts Proposition \ref{6.5}. 

Assume that $r_E =0$. Since $v(E)^2=0$, we have $n_E = 0$. 
Since there is an $\bb R$ divisor $\beta = b L$ such that $b < 0$, 
$E$ is a torsion sheaf by Lemma \ref{6.4} (2).  
Since $n_E =0$, $\dim \mr{Supp}(E)=0$. By the assumption $\Hom ^0_X(E,E)=\bb C$, 
$E$ is $\mca O_x$ for some $x \in X$.  
\end{proof}

Now we are ready to prove an easy consequence of Theorem \ref{6.6}.

\begin{cor}\label{6.7}(= Theorem \ref{thm1})
Let $(X,L_X)$ and $(Y,L_Y)$ be generic K3 and let $\Phi : D(Y) \to D(X)$ be an equivalence. 
If $\Phi _* (U(Y) )= U(X)$, then $\Phi $ can be written in the following way:
\[
\Phi (?)  = M \otimes  f_* (?)[n] ,
\] 
where $M$ is a line bundle on $X$, $f$ is an isomorphism $f:Y \to X$ and $n \in \bb Z$. 
\end{cor}

\begin{proof}
Let $E_y$ be $\Phi (\mca O_y)$ for an arbitrary closed point $y \in Y$. 
Since $\Phi _* (U(Y))= U(X)$, $E_y$ is $\mca O_x [n_y]$ ($n_y \in \bb Z$) for some $x \in X$ by Theorem \ref{6.6}.  
In addition the phase of $E_y$ is constant. 
So $[n_y]$ is also constant. 
Thus $E_y$ is given by $\mca O_{f(y)}[n]$. By \cite[Corollary 5.23]{Huy}, we complete the proof. 
\end{proof}

Here we define the subgroup $\Aut (D(X), U(X))$ of $\Aut (D(X))$:
\[
\Aut (D(X), U(X)):=\{ \Phi \in \Aut (\mca D) | \Phi _*(U(X))= U(X) \}. 
\]

Thus we obtain the following statement:

\begin{cor}\label{group}
Notations being as above, 
we have 
\[
\Aut (D(X), U(X)) = \mr{Tri}(X),
\]
where $\mr{Tri}(X)$ is the subgroup generated by shifts, tensor products of line bundles and automorphisms. 
\end{cor}

We remark that $\mr{Tri}(X)$ is actually written by $(\Aut (X) \ltimes \mr{Pic}(X) ) \times \bb Z[1]$. 

\begin{proof}
If $\Phi$ is in the right hand side, $\Phi (\mca O_x) = \mca O_y [n]$ for some $y \in X$ and $n \in \bb Z$. Thus $\Phi _*(U(X))=U(X)$. 
Conversely, if $\Phi$ is in the left hand side, $\Phi $ is in the right hand side by Corollary \ref{6.7}. 
\end{proof}

\begin{rmk}\label{6.8}
Throughout this remark, we assume that $A$ and $A'$ are abelian surfaces. 
Similarly to the case of K3 surfaces, we can construct $U(A)$. 
Hence $\Stab (A)$ is nonempty. 
In particular $\Stabd (A) = U(A)$ since $D(A)$ has no spherical objects (cf. \cite[Section 15]{Bri2}). 
In addition, the set of good stability conditions is equal to $U(A)$ (and thus is connected) by the result of \cite[Theorem 3.15]{HMS}. 
The property ``good'' preserved by any equivalence $\Phi : D(A') \to D(A)$. 
Hence for any equivalence $\Phi :D(A') \to D(A)$, $\Phi_* (U(A')) = U(A)$. 
Thus we have 
\[
\Aut (D(A),U(A)) = \Aut (D(A)).
\]
\end{rmk}

\begin{flushright}
\begin{tabular}{ll}
Kotaro Kawatani\\
Department of Mathematics \\
Graduate School of Science \\
Osaka University\\
Toyonaka 563-0043, Japan\\
kawatani@cr.math.sci.osaka-u.ac.jp\\
\end{tabular}

\end{flushright}


\begin{thebibliography}{ABCD}

\bibitem[1]{B-M}
{T. Bridgeland and A. Maciocia}, 
{\em Complex surfaces with equivalent derived categories}, 
{Math. Z.} \textbf{236} (2001), 677--697. 

\bibitem[2]{Bri}
{T. Bridgeland}, 
{\em Stability conditions on triangulated categories},
{Ann. of Math.} \textbf{166} (2007), 317--345.

\bibitem[3]{Bri2}
{T. Bridgeland},
{\em Stability conditions on K3 surfaces},
{Duke Math. J.} \textbf{141} (2008), 241--291. 

\bibitem[4]{Muk}
{S. Mukai},
{\em On the moduli space of bundles on K3 surfaces I},
in: {\em Vector bundles on algebraic varieties} (Bombay, 1984), 
{Tata Inst. Fund. Res. Stud. Math.}, 11, {Tata Inst. Fund. Res., Bombay, 1987},
341--413. 

\bibitem[5]{HLOY}
{S. Hosono, B. H. Lian, K. Oguiso and S-T. Yau},
{\em Fourier-Mukai partners of a K3 surfaces of Picard number one}, 
in: {Vector bundles and representation theory}, 
{Contemp. Math.}, 322, {Amer. Math. Soc.}, {Providence, RI}, 2003, 
43--55. 

\bibitem[6]{HL}
{D. Huybrechts and M. Lehn},
{The geometry of moduli spaces of sheaves},
{Aspects of Mathematics}, 1997.

\bibitem[7]{HMS}
{D. Huybrechts, E. Macri and P. Stellari},
{\em Stability conditions for generic K3 categories},
{Compositio Math.} \textbf{144} (2008),134--162.

\bibitem[8]{Huy}
{D. Huybrechts},
{Fourier-Mukai transformations in Algebraic Geometry},
{Oxford science publications}, 2006.

\bibitem[9]{Kaw}
{Y. Kawamata}, 
{\em $D$-Equivalence and $K$-Equivalence}. 
{J. Diff. Geometry}. \textbf{61} (2002), 147--171.   



\bibitem[10]{ST}
{P. Seidel and R. Thomas},
{\em Braid group actions on derived categories of coherent sheaves},
{Duke Math. J.} \textbf{108} (2001), 37--108. 

\bibitem[11]{Ste}
{P. Stellari}, 
{\em Some remarks about the FM-partners of K3 surfaces with Picard numbers 1 and 2}, 
{Geom. Dedicata} \textbf{108} (2004), 1--13.

\bibitem[12]{Tod}
{Y. Toda}, 
{\em Moduli stacks and invariants of semistable objects on K3 surfaces}, 
{Adv. Math.} \textbf{217} (2008), 2736--2781. 
\end{thebibliography}
\end{document}